\documentclass[aps, 12pt, amsfonts, amsmath, amssymb, nofootinbib]{revtex4}
\usepackage{amsthm,mathtools,calc}
\usepackage{tikz}
\usepackage{tkz-euclide}
\usetkzobj{all}
\usetikzlibrary{positioning,matrix,arrows,shapes.misc}
\usepackage{hyperref}	
\usepackage{breakurl}
\usetikzlibrary{decorations.pathreplacing, decorations.markings}

\usepackage[backgroundcolor=green,shadow]{todonotes}

\usepackage{verbatim} 

\begin{document}
	\newcommand{\N}{\mathbb{N}}
	\newcommand{\Z}{\mathbb{Z}}
	\newcommand{\Q}{\mathbb{Q}} 
	\newcommand{\C}{\mathbb{C}}
	\newcommand{\R}{\mathbb{R}}
	\newcommand{\K}{\Bbbk}
	\newcommand{\op}[1]{\text{#1}}
	\newcommand{\tp}{\otimes}
	\newcommand{\no}{\nonumber}
	\newcommand{\KK}{\mathfrak{K}}
	\newcommand{\feyn}{\textbf{F}}
	\newcommand{\encirc}[1]{\begin{tikzpicture}[baseline=(C.base)] \node[draw,circle,inner sep=0.5pt](C) {\footnotesize #1}; \end{tikzpicture}}
	\newcommand{\encircs}[1]{\begin{tikzpicture}[baseline=(C.base)] \node[draw,circle,inner sep=0.1pt](C) {\tiny #1}; \end{tikzpicture}}
	\newcommand{\oP}{\mathcal{P}}	
	\newcommand{\cCo}{\ensuremath{\mathsf{Cor}}}
	\newcommand{\Cor}{\ensuremath{\mathsf{Cor}}}
	\newcommand{\DCor}{\ensuremath{\mathsf{DCor}}}	
	\newcommand{\PrDCor}{\ensuremath{\mathsf{Pr_{DCor}}}}
	\newcommand{\PrSigma}{\ensuremath{\mathsf{Pr_{\Sigma}}}}
	\newcommand{\set}[2]{\left\{ #1 \ | \ #2 \right\} }	
	\newcommand{\cyc}[1]{\left(\hspace{-.6ex}\left( #1 \right)\hspace{-.6ex}\right)}	
	\newcommand{\cycu}[1]{\left(\hspace{-.3ex}\left( #1 \right)\hspace{-.3ex}\right)}	
	\newcommand{\cycm}[1]{\left(\hspace{-.7ex}\left( #1 \right)\hspace{-.7ex}\right)}	
	\newcommand{\cycb}[1]{\left(\hspace{-.9ex}\left( #1 \right)\hspace{-.9ex}\right)}	
	\newcommand{\ooo}[2]{\sideset{_{#1}}{_{#2}}{\mathop{\circ}}}	
	\newcommand{\dg}[1]{{\left| #1 \right|}}	
	\newcommand{\id}{1}	
	\newcommand{\oQC}{\mathcal{QC}}	
	\newcommand{\oQO}{\mathcal{QO}}	
	\newcommand{\oQOC}{\mathcal{QOC}}	
	\newcommand{\oOC}{\mathcal{OC}}	
	\newcommand{\oAss}{\mathcal{A}\mathit{ss}}	
	\newcommand{\oCom}{\mathcal{C}\mathit{om}}	
	\newcommand{\oMod}[1]{\textbf{Mod}\left(#1\right)}	
	\newcommand{\Span}{\mathrm{Span}}	
	\newcommand{\gr}[1]{\mathsf{#1}}	
	\newcommand{\oT}{\mathcal{T}}	
	\newcommand{\oEnd}[1]{{\mathcal{E}_{#1}}}	
	\newcommand{\oooo}[4]{\prescript{#3}{#1}{\circ}^{#4}_{#2}}	
	\newcommand{\cc}{\mathbf{c}}	
	\newcommand{\cd}{\mathbf{d}}	
	\newcommand{\co}{\mathbf{o}}	
	\newcommand{\cobar}{\mathbf{C}}	
	\newcommand{\comp}{\circ}	
	\newcommand{\uo}{\underline{o}}	
	\newcommand{\uc}{\underline{c}}	
	\newcommand{\ob}{{\overline{b}}} 
	\newcommand{\ot}{\otimes} 
	\theoremstyle{plain}
	\newtheorem{theorem}{Theorem}
	\newtheorem{lemma}[theorem]{Lemma}
	\newtheorem{proposition}[theorem]{Proposition}
	\newtheorem{corollary}[theorem]{Corollary}
	\theoremstyle{definition}
	\newtheorem{definition}[theorem]{Definition}
	\theoremstyle{remark}
	\newtheorem{example}[theorem]{Example}
	\newtheorem{remark}[theorem]{Remark}
	\newtheorem{conjecture}[theorem]{Conjecture}
	

	
	\newcommand{\PICGraph}{
		\begin{tikzpicture}[scale=0.7]
		\draw (-1.2,0) .. controls +(0,0) .. (0,0) node[below]{${V_1}$} .. controls +(2,1) .. (4,0) node[below]{${V_2}$};
		\draw (0,0) .. controls +(2,-1) .. (4,0) -- (5,1) node[right]{$\scriptscriptstyle{l_3}$};
		\draw (4,0) -- (5,-1) node[right]{$\scriptscriptstyle{l_2}$};
		\draw (2,5) node[left]{$\scriptscriptstyle{l_4}$} -- (2,4) node[right]{${V_3}$} -- (0,0);
		\draw (2,4) -- (4,0);
		\node at (-0.8,0.9) [below] {$\scriptscriptstyle{h_1=l_1}$};
		\node at (1,-0.35) [below] {$\scriptscriptstyle{h_2}$};
		\node at (1,0.5) [above] {$\scriptscriptstyle{h_3}$};
		\node at (0.6,1) [above left] {$\scriptscriptstyle{h_4}$};
		\draw (2,0.625) -- (2,0.875);
		\draw (2,-0.625) -- (2,-0.875);
		\draw (0.9,2.1) -- (1.2,1.9);
		\draw (2.9,1.9) -- (3.2,2.1);
		
		\end{tikzpicture}\qquad\qquad\qquad
		\begin{tikzpicture}[scale=0.7]
		\draw [decoration={markings, mark=at position 0.2 with {\arrow[scale=1.5,>=stealth]{>}}}, postaction={decorate}] 
		[decoration={markings, mark=at position 0.55 with {\arrow[scale=1.5,>=stealth]{>}}}, postaction={decorate}] 
		[decoration={markings, mark=at position 0.95 with {\arrow[scale=1.5,>=stealth]{>}}}, postaction={decorate}] (-1.2,0) .. controls +(0,0) .. (0,0) node[below]{${G_1}$} .. controls +(2,1) .. (4,0) node[below]{${G_2}$};
		\draw [decoration={markings, mark=at position 0.33 with {\arrow[scale=1.5,>=stealth]{>}}}, postaction={decorate}] 
		[decoration={markings, mark=at position 0.7 with {\arrow[scale=1.5,>=stealth]{>}}}, postaction={decorate}]
		[decoration={markings, mark=at position 0.83 with {\arrow[scale=1.5,>=stealth]{<}}}, postaction={decorate}] (0,0) .. controls +(2,-1) .. (4,0) -- (5,1) node[right]{$\scriptscriptstyle{l_3}$};
		\draw [decoration={markings, mark=at position 0.85 with {\arrow[scale=1.5,>=stealth]{>}}}, postaction={decorate}] (4,0) -- (5,-1) node[right]{$\scriptscriptstyle{l_2}$};
		\draw [decoration={markings, mark=at position 0.08 with {\arrow[scale=1.5,>=stealth]{<}}}, postaction={decorate}] 
		[decoration={markings, mark=at position 0.55 with {\arrow[scale=1.5,>=stealth]{>}}}, postaction={decorate}]
		[decoration={markings, mark=at position 0.95 with {\arrow[scale=1.5,>=stealth]{>}}}, postaction={decorate}] (2,5) node[left]{$\scriptscriptstyle{l_4}$} -- (2,4) node[right]{${G_3}$} -- (0,0);
		\draw [decoration={markings, mark=at position 0.45 with {\arrow[scale=1.5,>=stealth]{>}}}, postaction={decorate}] 
		[decoration={markings, mark=at position 0.93 with {\arrow[scale=1.5,>=stealth]{>}}}, postaction={decorate}] (2,4) -- (4,0);
		\node at (-0.8,0.9) [below] {$\scriptscriptstyle{h_1=l_1}$};
		\node at (1,-0.35) [below] {$\scriptscriptstyle{h_2}$};
		\node at (1,0.5) [above] {$\scriptscriptstyle{h_3}$};
		\node at (0.6,1) [above left] {$\scriptscriptstyle{h_4}$};
		\draw (2,0.625) -- (2,0.875);
		\draw (2,-0.625) -- (2,-0.875);
		\draw (0.9,2.1) -- (1.2,1.9);
		\draw (2.9,1.9) -- (3.2,2.1);
		\end{tikzpicture}
	}
	
	
	\newcommand{\PICclosed}{
		\begin{tikzpicture}[scale=0.5]
		\draw (6,4.5) node{$\scriptstyle{ }$} ellipse (0.2 and 0.5); 
		\draw (0,6) .. controls +(2,-0) .. (3,5) .. controls +(1,-0.75) .. (6,4);
		\draw (0,6.5) node{$\scriptstyle{ }$} ellipse (0.2 and 0.5); 
		\node[text width=1cm] at (-0.25,6.5){$\scriptstyle{C_1}$};
		\draw (0,7) .. controls +(2,-0) .. (3,8) .. controls +(1,0.75) .. (6,9);
		\draw (6,8.5) node{$\scriptstyle{ }$} ellipse (0.2 and 0.5); 
		\node[text width=1cm] at (7.5,8.5){$\scriptstyle{D_1}$};
		\draw (6,6.5) node{$\scriptstyle{ }$} ellipse (0.2 and 0.5); 
		\draw (6,5) .. controls +(-1.5,0) and +(-1.5,0) .. (6,6);
		\draw (6,7) .. controls +(-1.5,0) and +(-1.5,0) .. (6,8);
		\draw (-0.217+3.2,6.5) arc (210:330:0.5);
		\draw (-.13+3.2,-.115+6.5) arc (150:30:0.4);
		
		\draw (9,2.5) node{$\scriptstyle{ }$} ellipse (0.2 and 0.5); 
		\node[text width=1cm] at (8.75,2.5){$\scriptstyle{C_2}$};
		\draw (9,4.5) node{$\scriptstyle{ }$} ellipse (0.2 and 0.5);
		\draw (9,2)  .. controls +(1.3,.15) .. (10.8,2.5).. controls +(0.5,0.45) .. (11.9,3) .. controls +(0.3,0) .. (12.1,3).. controls +(0.45, -0.05)  .. (13.2,2.5) .. controls +(0.9,-0.5) .. (15,2);
		\draw (15,2.5) node{$\scriptstyle{ }$} ellipse (0.2 and 0.5); 
		\draw (15,6.5) node{$\scriptstyle{ }$} ellipse (0.2 and 0.5); 
		\draw (9,7)  .. controls +(1.3,-.15) .. (10.8,6.5).. controls +(0.5,-0.45) .. (11.9,6) .. controls +(0.3,0) .. (12.1,6).. controls +(0.45, 0.05)  .. (13.2,6.5) .. controls +(0.9,0.5) .. (15,7);
		\draw (9,6.5) node{$\scriptstyle{ }$} ellipse (0.2 and 0.5); 
		\draw (9,5) .. controls +(1.5,0) and +(1.5,0) .. (9,6);
		\draw (9,3) .. controls +(1.5,0) and +(1.5,0) .. (9,4);
		\draw (15,3) .. controls +(-1.7,0) and +(-1.7,0) .. (15,6);
		\draw (8.6-0.217+2.5,4.5) arc (210:330:0.5);
		\draw (8.6-.13+2.5,-.115+4.5) arc (150:30:0.4);
		\draw (10-0.217+2.5,4.5) arc (210:330:0.5);
		\draw (10-.13+2.5,-.115+4.5) arc (150:30:0.4);
		
		\node[text width=1cm] at (7.4,6.5){$\scriptstyle{b_1}$};
		\node[text width=1cm] at (7.4,4.5){$\scriptstyle{b_2}$};
		\node[text width=1cm] at (9,6.5){$\scriptstyle{a_1}$};
		\node[text width=1cm] at (9,4.5){$\scriptstyle{a_2}$};
		
		\draw[dashed] (6.2,4)--(8.9,4);
		\draw[dashed] (6.2,5)--(8.9,5);
		\draw[dashed] (6.2,6)--(8.9,6);
		\draw[dashed] (6.2,7)--(8.9,7);
		
		\draw [decorate,decoration={brace,amplitude=6pt},xshift=15.5cm,yshift=0pt]
		(0,7) -- (0,2) node [midway,right,xshift=.1cm] {$\scriptstyle{D_2}$};    
		\end{tikzpicture}
	}
	
	
	\newcommand{\PICShuffle}{
		\begin{tikzpicture}[scale=0.344]
		
		\draw [decoration={markings, mark=at position 1 with {\arrow[scale=2,>=stealth]{>}}}, postaction={decorate}] (-4.6,5.4) arc (165:212:5);
		\node[text width=1] at (-3.2,8.4){$\tilde{\cd_j}$};
		\draw 
		[decoration={markings, mark=at position 0.11 with {\filldraw circle [radius=0.06] node(x)[below left]{${y}$};}}, postaction={decorate}][decoration={markings, mark=at position 0.25 with {\filldraw circle [radius=0.06] node(x_1){};}}, postaction={decorate}]
		[decoration={markings, mark=at position 0.39 with {\filldraw circle [radius=0.06] node(x_2){};}}, postaction={decorate}]
		[decoration={markings, mark=at position 0.75 with {\filldraw circle [radius=0.06] node(x_n1){};}}, postaction={decorate}]
		[decoration={markings, mark=at position 0.91 with {\filldraw circle [radius=0.06] node(x_n){};}}, postaction={decorate}]
		(0,4) circle [radius=4.1];
		
		\node[below] at (x_1){$\scriptstyle{y_1}$};
		\node[below right] at (x_2){$\scriptstyle{y_2}$};
		\node[above] at (x_n1){$\scriptstyle{y_{m-1}}$};
		\node[left] at (x_n){$\scriptstyle{y_m}$};
		
		\draw 
		[decoration={markings, mark=at position 0.25 with {\draw circle [radius=0.07] node(y_m1){};}}, postaction={decorate}]
		[decoration={markings, mark=at position 0.4 with {\draw circle [radius=0.07] node(y_m){};}}, postaction={decorate}]
		[decoration={markings, mark=at position 0.51 with {\draw circle [radius=0.07] node(y)[right]{${x}$};}}, postaction={decorate}]
		[decoration={markings, mark=at position 0.63 with {\draw circle [radius=0.07] node(y_1){};}}, postaction={decorate}]
		[decoration={markings, mark=at position 0.76 with {\draw circle [radius=0.07] node(y_2){};}}, postaction={decorate}]
		[decoration={markings, mark=at position 0.9 with {\draw circle [radius=0.07] node(y_3){};}}, postaction={decorate}]
		(12,4) circle [radius=4.1];
		\node[below] at (y_m1){$\scriptstyle{x_{n-1}}$};
		\node[below right] at (y_m){$\scriptstyle{x_n}$};
		\node[above right] at (y_1){$\scriptstyle{x_1}$};
		\node[above] at (y_2){$\scriptstyle{x_2}$};
		\node[above left] at (y_3){$\scriptstyle{x_3}$};
		
		\draw [decoration={markings, mark=at position 1 with {\arrow[scale=2,>=stealth]{>}}}, postaction={decorate}] (16.5,2.4) arc (-20:35:5);
		\node[text width=1] at (14.2,8.4){$\cc_i$};
		\draw[dashed] (x.north east)--(y.west); 
		\end{tikzpicture}}
	
	\newcommand{\Openfrob}{
		\begin{tikzpicture}[scale=0.48]
		\draw 
		[decoration={markings, mark=at position 0.1 with {\filldraw circle [radius=0.06] node(a10)[above right]{};}}, postaction={decorate}]
		[decoration={markings, mark=at position 0.3 with {\filldraw circle [radius=0.06] node[above]{};}}, postaction={decorate}]
		[decoration={markings, mark=at position 0.44 with {\filldraw circle [radius=0.06] node[left]{};}}, postaction={decorate}]
		[decoration={markings, mark=at position 0.6 with {\filldraw circle [radius=0.06] node(a13)[below left]{};}}, postaction={decorate}]
		[decoration={markings, mark=at position 0.78 with {\filldraw circle [radius=0.06] node[below]{};}}, postaction={decorate}]
		[decoration={markings, mark=at position 0.88 with {\filldraw circle [radius=0.06] node[below right]{};}}, postaction={decorate}]
		(7,8-1) ellipse (0.8 and 1.3);
		\draw [decoration={markings, mark=at position 1 with {\arrow[scale=1.5,>=stealth]{>}}}, postaction={decorate}] (7.5,8-1) arc (15:40:2.5);
		\draw
		[decoration={markings, mark=at position 0.2 with {\filldraw circle [radius=0.06] node[above]{};}}, postaction={decorate}]
		[decoration={markings, mark=at position 0.5 with {\filldraw circle [radius=0.06] node[left]{};}}, postaction={decorate}]
		[decoration={markings, mark=at position 0.72 with {\filldraw circle [radius=0.06] node[below]{};}}, postaction={decorate}]
		[decoration={markings, mark=at position 0.92 with {\filldraw circle [radius=0.06] node(a17)[below right]{};}}, postaction={decorate}]
		(7,3.5) ellipse (0.5 and 1); 
		\draw [decoration={markings, mark=at position 1 with {\arrow[scale=1.5,>=stealth]{>}}}, postaction={decorate}] (7.25,3.55) arc (20:36:2.7);
		\draw 
		[decoration={markings, mark=at position 0.2 with {\filldraw circle [radius=0.06] node[above right]{};}}, postaction={decorate}]
		[decoration={markings, mark=at position 0.5 with {\filldraw circle [radius=0.06] node[left]{};}}, postaction={decorate}]
		[decoration={markings, mark=at position 0.9 with {\filldraw circle [radius=0.06] node(a21)[below right]{};}}, postaction={decorate}]
		(7,0.5) ellipse (0.5 and 1); 
		\draw [decoration={markings, mark=at position 1 with {\arrow[scale=1.5,>=stealth]{>}}}, postaction={decorate}] (7.25,0.05+0.5) arc (20:36:2.7);
		
		\draw (7,1+0.5) .. controls +(-1.5,0) and +(-1.5,0) .. (7,2.5);
		\draw (7,4.5) .. controls +(-1.5,0) and +(-1.5,0) .. (7,8-1.3-1);
		
		\draw 
		[decoration={markings, mark=at position 0.1 with {\draw circle [radius=0.07] node[right]{};}}, postaction={decorate}]
		[decoration={markings, mark=at position 0.25 with {\draw circle [radius=0.07] node(a1)[above]{};}}, postaction={decorate}]
		[decoration={markings, mark=at position 0.4 with {\draw circle [radius=0.07] node[left]{};}}, postaction={decorate}]
		[decoration={markings, mark=at position 0.55 with {\draw circle [radius=0.07] node[left]{};}}, postaction={decorate}]
		[decoration={markings, mark=at position 0.7 with {\draw circle [radius=0.07] node(a4)[left]{};}}, postaction={decorate}]
		[decoration={markings, mark=at position 0.87 with {\draw circle [radius=0.07] node(a5)[right]{};}}, postaction={decorate}]
		(11-14,7-0.5) ellipse (0.8 and 1.3);
		\draw [decoration={markings, mark=at position 1 with {\arrow[scale=1.5,>=stealth]{>}}}, postaction={decorate}] (11.5-14,7-0.5) arc (15:40:2.5);
		\draw 
		[decoration={markings, mark=at position 0.5 with {\draw circle [radius=0.07] node[left]{};}}, postaction={decorate}]
		[decoration={markings, mark=at position 0.90 with {\draw circle [radius=0.07] node[right]{};}}, postaction={decorate}]
		(11-14,3-0.5) ellipse (0.5 and 1); 
		\draw [decoration={markings, mark=at position 1 with {\arrow[scale=1.5,>=stealth]{>}}}, postaction={decorate}] (11.25-14,3.05-0.5) arc (20:36:2.7);
		\draw (11-14,4-0.5) .. controls +(1.3,0) and +(1.3,0) .. (11-14,7-1.3-0.5);
		
		\draw (13.-0.417-13,3.5-0.1+1.5) arc (210:330:0.82);
		\draw (13.-.301-13+0.03,-.115+3.5-0.19+1.5) arc (150:30:0.6);
		\draw (4.1-0.417-1,4.5-0.1+2) arc (210:330:0.82);
		\draw (4.1-.301-1+0.03,-.115+4.5-0.19+2) arc (150:30:0.6);
		\draw (4.1-0.417-0.5,2.5-0.1+1) arc (210:330:0.82);
		\draw (4.1-.301-0.45,-.115+2.5-0.19+1) arc (150:30:0.6);
		
		\draw (11-14,2-0.5) .. controls +(5.1,0) and +(0.1,-0.4) .. (7,-1+0.5);
		\draw (11-14,7+1.3-0.5) .. controls +(5.1,-0.7) and +(0.1,0.2) .. (7,8+1.3-1);
		
		\node[text width=1] at (8+0.2,7.3){$\cd_1$};
		\node[text width=1] at (7.8+0.1,3.8){$\cd_2$};
		\node[text width=1] at (7.8+0.1,0.4){$\cd_3$};
		\node[text width=1] at (-4.9,7.1){$\cc_1$};
		\node[text width=1] at (-4.5,3.4){$\cc_2$};
		
		\end{tikzpicture}}


	\title{Properads and Homotopy Algebras Related to Surfaces \vspace{1.0cm}}
	\author{Martin Doubek} \author{Branislav Jur\v co}\email{branislav.jurco@gmail.com} \affiliation{Charles University, Faculty of Mathematics and Physics, Sokolovsk\'a 83, 186 75 Prague, Czech Republic}\author{Lada Peksov\'a}\email{lada.peksova@gmail.com} 
	
	\affiliation{Charles University, Faculty of Mathematics and Physics, Sokolovsk\'a 83, 186 75 Prague, Czech Republic\\and\\
		Georg-August University, Faculty of Mathematics, Bunsenstra\ss e 3-5, 370 73 G\"ottingen, Germany}

	\begin{abstract}
		\vspace{1.7cm}
		\centerline{\bf Abstract \vspace{0.4cm}}
		Starting from a biased definition of a properad, we describe explicitly algebras over the cobar construction of a properad.  Equivalent description in terms of solutions of generalized master equations, which can be interpreted as homological differential operators, is explained from the properadic point of view. This is parallel to Barannikov's theory for modular operads. In addition to the well known IBL-homotopy algebras, the examples include their associative analogues, which we call $IBA$-homotopy algebras, and a combination of the above two.
		
	\end{abstract}
	
	\maketitle
	\thispagestyle{empty}
	\centerline{\it{Dedicated to the memory of Martin Doubek.}}
	\newpage
	\tableofcontents
	\section{Introduction} 
	
	Operads are objects that model operations with several inputs and one output. As such, they can be generalized in the context of graphs in two possible ways. The undirected graphs with several inputs lead to the notion of cyclic and modular operads, whereas the connected directed graphs with several inputs and several outputs lead to the notion of properads. Both include examples, which can be interpreted  in terms of 2-dimensional surfaces, with boundaries and punctures\footnote{Punctures can be in the interior and/or on the boundaries, i.e., describing both open and/or closed strings}. For modular operads, a detailed discussion can be found in \cite{DJM} by M\"unster and the two first authors. The aim of this article is to continue this work but this time for properads.
	
	The properads  were first introduced in \cite{Bruno} by Vallette as connected parts of PROPs. In \cite{Bruno} he gives both an unbiased as well as a biased definition.  Here we use a biased definition, which is at closest to the one in \cite{Hackney}. In our definition, a properad is indexed by two finite sets. Our main examples are the closed (commutative) Frobenius properad, the open (associative) Frobenius properad, open-closed Frobenius properad and their cobar complexes, which we use to study the corresponding homotopy algebras.
	
	{Barannikov \cite{BarannikovModopBV} showed how an algebra over the cobar construction over a modular operad can equivalently be described as a solution of a master equation for certain generalized BV algebra\footnote{The cobar construction in the category of modular operads is also known as the Feynman transform. The Feynman transform produces, out of a modular operad, a twisted modular operad. Although this name might seem to be an unfortunate choice, it is widely used in the literature and we will use it occasionally.}. Here we give an analogous description for properads.
		In this paper, we consider the construction of the cobar complex in the same manner as \cite{MarklShneiderStasheff}. In short, the cobar complex of a properad is a free properad over its suspended linear dual equipped with the differential induced by the duals of the structure operations.}
	
	The paper is organized as follows. Sect. 2 contains conventions and notations used through the paper. In Sect. 3, we introduce properads along with our main examples, the closed (commutative) Frobenius properad, the open (associative) Frobenius operad, the open-closed Frobenius properad and the endomorphism properad.  Further, we recall the cobar construction over a properad. Finally, we give an analog of Barannikov's theory for algebras over the cobar complex.
	In Sect. 4, we give explicit descriptions of algebras over the cobar complexes of (open, closed and open-closed) Frobenius properads. In the closed case, we recover the  well-known result that the corresponding algebras are $IBL_\infty$-algebras, cf. \cite{Drummond}, \cite{Cieliebak ibl}\footnote{For a recent, intriguing description of $IBL_\infty$-algebras, see \cite{Markl}}. The explicit descriptions include one in terms of operations with $m$ inputs and $n$ outputs labeled by a genus $g$ and one in terms of a ``homological  differential operator" on formal power series in $\tau$, the formal variable of degree 0, with values in the symmetric algebra over the underlying dg graded vector space.  Finally, we give the corresponding description in the closed and open-closed case.
	
	In this paper, we do not discuss morphisms and minimal models for our homotopy algebras\footnote{For $IBL_\infty$-algebras this is discussed in detail in \cite{Cieliebak ibl}}. This will be done elsewhere. 
	
	We finish this Introduction with the following remark. As noted in \cite{Cieliebak ibl}, the algebraic framework of $IBL_\infty$-algebras is used for three
	related but different purposes: (equivariant) string topology, symplectic field theory and Lagrangian Floer theory of higher genus. Also, the quantum open-closed string field theory can be formulated in the language of (some particular) $IBL_\infty$-algebras and their morphisms \cite{qocha}, \cite{Ivo-Korbi}.
	We hope that also the other homotopy algebras discussed in this paper might possibly find similar applications.
	
	The last two authors are responsible for all possible mistakes and errors.

	\section{Conventions and notation}
	
	\begin{enumerate}
		\item $\N$ is the set of positive integers, $\N_0:=\N\cup\{0\}$.
		\item $\K$ is a field of characteristic $0$. The multiplication in $\K$ will be either denoted . or omitted. All (dg) vector spaces are considered over $\K$.
		\item Dg vectors spaces have differential of degree $+1$.
		Morphisms of dg vector spaces are linear maps commuting with differentials.
		\item $\sqcup$ is \emph{disjoint} union.
		Whenever $A\sqcup B$ appears, $A,B$ are automatically assumed disjoint.
		\item $\xrightarrow{\sim}$ or $\xrightarrow{\cong}$ denotes an iso (in particular a bijection).\item $\uparrow$ is suspension.
		\item $A^{\#}$ is the linear dual of $A$.
		\item $\Sigma_n$ is the symmetric group on $n$ elements.
		\item $[n]$ is the set $\lbrace 1, 2, \ldots n \rbrace$
	\end{enumerate}

	\section{Properads, Cobar Construction and Master Equation}\label{sec:operads}
	
	\subsection{Properads}
	Denote by $\Cor$ the category of finite sets and their
	isomorphisms (corollas). 
	
	\begin{definition} \label{DEFCorr}
		Denote by $\DCor:= \Cor \times \Cor$ the category of directed  corollas: the objects are pairs $(C,D)$ with $C$ and $D$ finite sets which are called the  outputs and inputs. 
		A morphism $(\rho, \sigma):(C,D)\to (C',D')$ is a pair of bijections $\rho: C\xrightarrow{\sim}C'$, $\sigma: D\xrightarrow{\sim}D'$.
	\end{definition}
	
	\begin{definition} \label{DEFModOp} 
		A properad  $\oP$ consists of a collection
		$$\set{\oP(C,D)}{(C,D)\in\DCor}$$ of dg vector spaces and two collections of degree $0$ morphisms of dg vector spaces
		\begin{gather*}
		\set{\oP(\rho,\sigma):\oP(C,D)\to\oP(C',D')}{(\rho, \sigma):(C,D)\to(C',D')}\\
		\set{\stackrel{\eta}{\ooo{B}{A}}:\oP(C_1,D_1\sqcup B)\tp\oP(C_2\sqcup A, D_2)\!\to\!\oP(C_1\sqcup C_2,D_1\sqcup D_2)}
		{\eta: B\xrightarrow{\sim}A }
		\end{gather*}	
		where $A,B$ are arbitrary isomorphic finite nonempty sets. 
		These data are required to satisfy the following axioms:
		\begin{enumerate}
			\item $\oP((\id_C,\id_D ))=\id_{\oP(C,D)}, \quad \oP((\rho\rho',\sigma'\sigma))=\oP((\rho,\sigma)) \ \oP((\rho',\sigma'))$
			\item $(\oP((\rho_1\sqcup\rho_2|_{C_2},\sigma_1|_{D_1}\sqcup\sigma_2)) \ \stackrel{\eta}{\ooo{B}{A}} = \stackrel{\rho_2\eta\sigma_1^{-1}}{\ooo{\sigma_1(B)}{\rho_2(A)}} \ (\oP((\rho_1,\sigma_1))\tp\oP((\rho_2,\sigma_2))$
			
			\item $\stackrel{\epsilon}{\ooo{B_2\sqcup B_3}{A_2\sqcup A_3}}\ (\stackrel{\tilde{\eta}}{\ooo{B_1}{A_1}}\tp\id) = \stackrel{\eta}{\ooo{B_1\sqcup B_3}{A_1\sqcup A_3}} \ (\id\tp\stackrel{\tilde{\epsilon}}{\ooo{B_2}{A_2}})$ 
			
			where $\tilde{\eta}, \tilde{\epsilon}$ are restrictions of $\eta, \epsilon$ to the pairs of nonempty sets $A_1, B_1$ and $A_2, B_2$, respectively. 
			
			For $A_1, B_1$ empty sets, 
			
			$\stackrel{\tilde{\epsilon}}{\ooo{B_2}{A_2}}\ (\stackrel{\eta}{\ooo{B_3}{A_3}}\tp\id) = \stackrel{\eta}{\ooo{B_3}{A_3}} \ (\id\tp\stackrel{\tilde{\epsilon}}{\ooo{B_2}{A_2}})$. 
			
			For $A_2, B_2$ empty sets,  
			
			$\stackrel{{\epsilon}}{\ooo{B_3}{A_3}}\ (\stackrel{\tilde{\eta}}{\ooo{B_1}{A_1}}\tp\id) = \stackrel{\tilde{\eta}}{\ooo{B_1}{A_1}} \ (\id\tp\stackrel{{\epsilon}}{\ooo{B_3}{A_3}})$. 
		\end{enumerate}
		whenever the expressions make sense.

		By $\PrDCor$ we will denote the category of properads with the obvious morphisms. 
	\end{definition}
	
	\begin{remark} \label{REMSModCyc}
		If we consider only Axiom $1.$, the resulting structure is called a $\Sigma$-bimodule.
		Obviously, by forgetting the composition map, a properad gives rise to its underlying $\Sigma$-module.
		
		All these notions are equivalent to their usual counterparts in \cite{Bruno}.
		For example, Axiom $1.$ stands for the left and right $\Sigma$-actions on $C,D$ respectively, $2.$ expresses the equivariance and $3.$ expresses the associativity of the structure maps.
	\end{remark}
	
	In this paper, we consider only properads such that the dg vector spaces $\oP(C,D)$ have an additional $\mathbb{N}_0$ grading by a degree which  will be denoted by $G$. The differential and both $\Sigma$-actions are assumed to preserve the degree $G$-components $\oP(C,D,G)$. For operations $\stackrel{\eta}{\ooo{B}{A}}$, we assume that they map the components with respective degrees $G_1$ and $G_2$ into the component of the degree $G(G_1,G_1,A, B, \eta)$ which is determined, in general, by the degrees $G_1$, $G_2$ sets $A$, $B$ and their identification $\eta$. Also, let us introduce $\chi := 2G + |C| + |D| -2$. Correspondingly, we will use the notation $\oP(C,D,\chi)$ for $\oP(C,D,G)$ with $2G=\chi - |C| - |D| +2\geq 0$. Having in mind the forthcoming examples, we refer to $\chi$ as the ``Euler characteristic''.
	
	We will assume the stability condition $\chi >0$, unless explicitly mentioned otherwise. In particular, this means that for $G=0$, $|C| +|D|\geq 3$ and for $G=1$, $|C| +|D|\geq 1$. For $G>1$, there is no restriction on the number of inputs and outputs. 
	
	Here we should mention that we use slightly different conventions as in \cite{Bruno}, where it is assumed that the sets $C$ and $D$ are always non-empty, i.e., there is always at least one input and one output. Also, in \cite{Bruno}, one input and one output are allowed for $G=0$. We will comment on this further when describing the cobar complex and algebras over it.
	
	It will prove useful to consider the skeletal version of properads. 
	\begin{definition}
		$\bf{\Sigma}$ is the skeleton of category $\DCor$ consisting of corollas of the form $([m],[n])$, $m,n\in \mathbb{N}_0$. ${\Sigma}$-bimodule is a functor from $\bf{\Sigma}$ to dg vector spaces.
	\end{definition}
	
	Before giving the next definition, let us introduce the following convenient notation. For $n\in \mathbb{N}_0$ and a set $\{a_1,a_2,\ldots\}$ of natural numbers, define  
	$$ n + \{a_1,a_2,\ldots\} := \{n+a_1,n+a_2,\ldots\}.$$
	Given $N\subset [n_1+|N|]$, and $M \subset [m_2+|M|]$ define bijections
	$$\rho_N : [n_1+|N|]- N \to n_2+[n_1],$$
	$$\rho_M: [m_2+|M|]- M \to m_1 + [m_2]$$ 
	by requiring them to be increasing.\footnote{The meaning of $n_2$ and $m_1$ will become clear from the next definition.} 
	
	\begin{definition} 
		Given a properad $\oP$ with structure morphisms $\stackrel{\eta}{\ooo{B}{A}}$, define $\bar\oP$ to consist of a collection
		$$\set{\bar\oP(m,n)}{([m],[n])\in\DCor}$$ of dg $\Sigma_m\times \Sigma_n$-bimodules  and a collection 
		\begin{gather*}
		\set{\stackrel{\xi}{\bar{\ooo{N}{M}}}:\bar\oP(m_1,n_1 + |N| )\tp\bar\oP(m_2  + |M|,n_2 ) \!\to\!\bar\oP(m_1+m_2, n_1+n_2)}{
			\xi: N\xrightarrow{\sim}M} 
		\end{gather*} 
		of a degree $0$ morphisms of dg vector spaces determined by formulas
		$$\bar\oP(m,n):=\oP([m],[n])$$
		$$\stackrel{\xi}{\bar{\ooo{N}{M}}}:= \oP(\kappa_1^{-1}\sqcup \rho_M\kappa_2^{-1}|_{C_2},\rho_N\lambda_1^{-1}|_{D_1}\sqcup \lambda_2^{-1} )\stackrel{\eta}{\ooo{B}{A}}(\oP(\kappa_1,\lambda_1)\tp\oP(\kappa_2,\lambda_2)),$$
		where $\kappa_1:[m_1]\to C_1$, $\lambda_1:[n_1 +|B|]\to D_1 \sqcup B$, $\kappa_2:[m_2 +|A|]\to C_2 \sqcup A$ and $\lambda_2:[n_2]\to D_2$ are arbitrary bijections such that $C_1\cap C_2 =D_1\cap D_2=\emptyset$ and $\xi=\kappa_2^{-1}\eta\lambda_1$. Also, $M=\kappa_2^{-1}A$  and $N=\lambda_1^{-1}B$.
	\end{definition}
	
	\begin{definition}  A shuffle $\sigma$ of type $(p, q)$ is an element of $\Sigma_{p+q}$ such that $\sigma(1)\! <\! \sigma(2)\! <\! \ldots\!\!<\! \sigma(p)$ and $\sigma(p + 1) < \ldots < \sigma(p+q)$. 
		Similarly an unshuffle $\rho$ of type $(p, q)$ is an element of $\Sigma_{p+q}$ such that we have $\rho(i_j) = j$ for some $i_1 < i_2 < \ldots < i_p$ , $i_{l+1} < \ldots < i_k$. Hence, $\rho$ is an unshuffle if $\rho^{-1}$ is a shuffle.
	\end{definition}
	
	\begin{remark}\label{poznamka s unshuffly} Obviously, the definition of $\stackrel{\xi}{\bar{\ooo{N}{M}}}$ doesn't depend on bijections $\kappa_1,\lambda_1, \kappa_2, \lambda_2$. Hence, sometimes, it might be useful, to make some simplifying choices of these. If, e.g., $C_1\cup C_2= [m], D_1\cup D_2= [n]$, $\kappa_1, \lambda_2$ as well as  $\lambda_1|_{[n_1 +|B|]-B}$ and $\kappa_2|_{[m_2 +|A|]-A}$ are increasing,  then $(\kappa_1^{-1}\sqcup \rho_M\kappa_2^{-1}|_{C_2})$ and $(\rho_N\lambda_1^{-1}|_{D_1}\sqcup \lambda_2^{-1})$ are $(m_1,m_2)$ and $(n_2,n_1)$-unshuffles, respectively.
	\end{remark}
	
	The operations $\stackrel{\xi}{\bar{\ooo{N}{M}}}$ satisfy properties analogous to the axioms of Definition \ref{DEFModOp}.
	Hence, we can introduce a new category $\PrSigma$ of $\Sigma$-bimodules with operations $\stackrel{\xi}{\bar{\ooo{N}{M}}}$. Obviously, categories $\PrDCor$ and $\PrSigma$ are equivalent. Although the axioms for operations  $\stackrel{\xi}{\bar{\ooo{N}{M}}}$ in  $\PrSigma$ is a way too complicated for practical purposes. Nevertheless, as we will see, the description of endomorphism properads $\oEnd{V}$ in the category $\PrSigma$ is nice and simple.
	
	Next, we give some examples:
	
	\begin{example}{The (closed) Frobenius properad $\mathcal{F}.$}\label{closed frob properad}
		For each $(C,D)\in \DCor$ and  $\chi>0$,\footnote{Recall, $\chi = 2G + |C| + |D| -2$.} put $\mathcal{F}(C,D,\chi) = {\Bbbk}$, i.e., the linear span on one generator $p_{C,D,\chi}$ in degree zero. The differential is trivial, as well as the $\Sigma$-bimodule structure. The operations $\stackrel{\eta}{\ooo{B}{A}}$ do not depend on sets $A,B$ and $\eta$,
		$$\stackrel{\eta}{\ooo{B}{A}}: p_{C_1,D_1\sqcup B,\chi_1}\otimes p_{C_2\sqcup A,D_2,\chi_2}\mapsto p_{C_1\sqcup C_2, D_1\sqcup D_2,\chi_1 +\chi_2}.$$ 
	\end{example}
	
	Geometrically, this properad consists of homeomorphism classes of 2-dimensional compact oriented surfaces with two kinds of labeled boundary components, the inputs and outputs. Here, $G=g$, is the geometric genus of the surface. Under the operation $\stackrel{\eta}{\ooo{B}{A}}$, we have $g = g_1 +g_2 + |A| -1$. Hence, the Euler characteristic $\chi$ is indeed additive as indicated in the formula above. Bijections act by relabeling  the inputs and outputs independently. The operation $\stackrel{\eta}{\ooo{B}{A}}$ for a non-trivial pair of sets $(A,B)$ consists of gluing surfaces along the inputs in $B$ and outputs in $A$ identified according to $\eta$. 
	
	\begin{figure}[h]
		\begin{center}
			\PICclosed
			\caption{ $\stackrel{\eta}{\ooo{B}{A}}$, where $A=\{a_1, a_2\}$, $B=\{b_1, b_2\}$ and $\eta(b_1)=a_1, \eta(b_2)=a_2$
			}
		\end{center}
	\end{figure}
	
	\begin{definition}
		A cycle in a set $C$ is an equivalence class $\cyc{x_1,\ldots,x_n}$ of an $n$-tuple $(x_1,\ldots,x_n)$ of several distinct elements of $C$ under the equivalence $(x_1,\ldots,x_n) \sim \tau(x_1,\ldots,x_n)$, where $\tau\in\Sigma_n$ is the cyclic permutation $\tau(i)=i+1$ for $1\leq i\leq n-1$ and $\tau(n)=1$.
		In other words,
		$$\cyc{x_1,\ldots,x_n} = \cdots = \cyc{x_{n-i+1},\ldots,x_n,x_1,\ldots,x_{n-i}} = \cdots = \cyc{x_2,\ldots,x_n,x_1}.$$
		We call $n$ the length of the cycle.
		We also admit the empty cycle $\cyc{}$, which is a cycle in any set.
	\end{definition}
	
	For a bijection $\rho:C\xrightarrow{\sim}D$ and a cycle $\cyc{x_1,\ldots,x_n}$ in $C$, define a cycle in $D$:
	$$\rho\cyc{x_1,\ldots,x_n} := \cyc{\rho(x_1),\ldots,\rho(x_n)}.$$
	
	\begin{example}{The open Frobenius properad $\mathcal{OF}$.}
		\begin{gather*}
		\mathcal{OF}(C,D,\chi) := \Span_{\K}\left\{ \{\cc_1, \cc_2, \ldots \cc_{p},\cd_1, \cd_2 \ldots,\cd_{q}\}^g \ |\ b_1, b_2\in\N,\ g\in\N_0\right\},
		\end{gather*}
		where $\cc_i,\cd_j$ are cycles in $C$ and $D$, respectively,  $\bigsqcup_{i=1}^{p}\cc_i=C,\bigsqcup_{j=1}^{q}\cd_j=D$, and for stable $\chi=2G+|C| +|D|-2>0$, with $G=2g +b$, $b=p+q$. Also,  $\{\cc_1, \cc_2, \ldots \cc_{p},\cd_1, \cd_2 \ldots,\cd_{q}\}^g$  is a symbol of degree $0$, formally being a pair consisting of $g\in\N_0$ and a set of cycles in $(C,D)$ with the above properties.
		We assume also $\cc_i\cap\cd_j=\emptyset$, $\cc_i\cap\cc_j=\emptyset$ and $\cd_i\cap\cd_j=\emptyset$ for all $i,j$. 
		
		For a pair of bijections $(\rho,\sigma):(C,D)\xrightarrow{\sim}(C',D')$, let 
		$$\mathcal{OF}(\rho,\sigma))(\{\cc_1, \cc_2, \ldots \cc_{p},\cd_1, \cd_2 \ldots,\cd_{q}\}^g) := \{(\rho(\cc_1), \ldots \rho(\cc_{p}),\sigma^{-1}(\cd_1),\ldots \sigma^{-1}(\cd_{q})\}^g.$$

		The formal definition of the operations $\stackrel{\eta}{\ooo{B}{A}}$ is a bit clumsy, so we refrain from it. Instead, note the following geometric interpretation: $\mathcal{OF}(C,D,\chi)$ is spanned by homeomorphism classes of 2-dimensional compact oriented stable surfaces with genus $g$, $p$ output boundaries and $q$ input boundaries.  The input boundaries can be permuted freely among themselves, as well as the output boundaries. We put $\chi=2(2g +b-1)+|C| +|D|-2$, i.e, $G=2g+b-1$, with $b=p +q$. 
		\begin{figure}[h]
			\begin{center}
				\Openfrob
				
			\end{center}
		\end{figure}
		
		The result of $\stackrel{\eta}{\ooo{B}{A}}$ is  obtained by (orientation preserving) gluing of two surfaces along the inputs in $B$ and outputs in $A$ identified according to $\eta$. Such a gluing creates a new surface which might contain mixed cycles, i.e., cycles containing both inputs and outputs. Such mixed cycles are subsequently split, within the resulting surface (and in an orientation preserving way), into pairs of cycles containing either inputs or outputs only. In the following, we will occasionally refer to the elements of boundary cycles as to segments.
		\\
		
		We hope that the following examples of such gluings and splittings will clarify geometric interpretation given above.
		We start with the simplest example of gluing two surfaces along one output and one input. 
		
		More specifically, we want to glue together an output segment $x$ of cycle $\cc_i = \cyc{x,x_1, x_2, \ldots x_n}$ of boundary $b_i$ together with an input segment $y$ of cycle $\tilde{\cd_j}=\cyc{y_1, y_2,\ldots y_m,y}$ of boundary $\tilde{b_j}$.\footnote{Note, that using the cyclic symmetry of the cycles, we can always move the segments $x$ and $y$ to these positions within the respective cycles.} According to the above description, there are two steps. In the first one a new mixed cycle $\cyc{y_1, y_2,\ldots y_m, x_1, x_2, \ldots x_n}$ is created. Hence, this new cycle is obtained by identifying of $x$ with $y$, removing the resulting point an joining the remaining parts of the original cycles, so that the resulting orientation is still compatible with the induced orientation of the boundaries. 

		\begin{figure}[h]
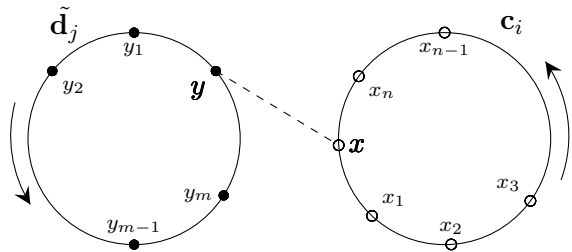

			\begin{center}
				\PICShuffle
				\caption{Connecting a segment $x$ from boundary $b_i$ with a segment $y$ from boundary $\tilde{b_j}$. The output segments are depicted as black circles and the input segments as white circles.}
				
			\end{center}
		\end{figure}
		
		However, we want to get again boundaries with inputs or outputs only. This leads to the second step, where we split the new cycle into two cycles of just outputs $\cyc{x_1, x_2, \ldots x_n}$ and inputs $\cyc{y_1, y_2,\ldots y_m}$.\footnote{This step may look bit trivial in this case but it gives a nontrivial result in the general case.} 
		\\

		Let us now turn our attention to the general case when on each of the two surfaces there are several segments on several boundaries which have to be glued together.
		Obviously, the formal description, as in the previous case, would be this time too complicated. In sake of simplicity, let us instead describe in words a simple algorithm how to glue the segments in order to obtain the mixed cycles (composed of both inputs and outputs), i.e., the ``new cycles" from step 1 above, and give one simple example to illustrate it.
		
		Obviously, we can consider only boundaries on which the segments that we are gluing together are positioned and ignore the rest. Let us choose one arbitrary segment of one of these boundaries which has to be glued.\footnote{It can be either an input or output segment.} Following the orientation of its cycle, we write down the segments of this cycle until we meet another segment which has to be glued to an another segment of a boundary on the other surface. We do not write down this segment nor its ``glued partner'', but instead we move to this partner along the gluing and continue in recording the segments according to the orientation of the partner's cycle. We continue this procedure until we get back to the point where we started. The recorded sequence gives a mixed cycle, the ``new cycle". To find all these mixed cycles, we choose another segment which wasn't written yet and start the procedure again.\footnote{Now already within the newly created surface.} 
		
		This gives us cycles with mixed outputs and inputs, but all of them could be split again into cycles of inputs and of outputs only by omitting the segments of the other type. We should be cautious with the following. If in course of this procedure an empty cycle arises, we have to split it too into an ``output'' and an ``input'' cycle. 
		\\
		
		An illustrative example could be gluing together segments $x_2$ with $y_6$, $x_3$ with $z_1$ and $x_7$ with $y_4$ of cycle $\cyc{x_1, x_2, \ldots x_8}$  of output segments of one surface and of cycles $\cyc{y_1, y_2,\ldots y_6}, \cyc{z_1, z_2, \ldots z_4}$
		of input segments of an another one.
		
		Let us choose one arbitrary segment, for example $y_1$. Following the orientation we write in a sequence $y_1,y_2,y_3$. The following segment $y_4$ is glued so we do not write it nor its glued partner $x_7$ but we continue from the position of $x_7$ according to orientation, i.e., with $x_8, x_1$. Then again, $x_2$ is glued with $y_6$ so we move to position of $y_6$ without recording this glued couple and continue according orientation. By this we get again into the position of $y_1$ where we started. One of the mixed cycles is therefore $\cyc{y_1,y_2,y_3,x_8,x_1}$. 
		
		To obtain another mixed cycle we choose for example $x_4$ and by following the orientation we get a beginning of the sequence $x_4, x_5,x_6$ which eventually gives us a mixed cycle $\cyc{x_4,x_5,x_6,y_5,z_2,z_3,z_4}$.
		
		These two mixed cycles are later split into cycles $\cyc{x_8,x_1}$, $\cyc{x_4,x_5,x_6}$ of input segments and into cycles $\cyc{y_1,y_2}$, $\cyc{y_5,z_2,z_3,z_4}$ of output segments.
		
		
	\end{example}
	\begin{remark}
		One can check that 
		
		i) The above algorithm is independent on the choices made.
		
		ii) 
		But now, the Euler characteristic, in contrary to the closed Frobenius properad, is not additive anymore. Concerning the genus of the resulting surface, it is given by a sum of genera of the original surfaces and the number of distinct pairs of boundaries which were ``glued together''. For instance, in the last illustrative example, there are only two distinct pairs of boundaries which were glued together although we glued together three pairs of segments.
	\end{remark}

	Finally, we can combine the above two properads in a rather simple way to obtain a $2$-colored properad $\mathcal{OCF}$, which we call open-closed Frobenius properad. 
	
	\begin{definition} Let $\DCor_2$ be the category of 2-colored
		directed corollas. The objects are pairs $((O_1,O_2,)(C_1,C_2),G)$, where $(O_1,O_2)$ and $(C_1,C_2)$ are pairs of finite sets and $G$ is a non-negative half-integer, i.e. of the form $G= \frac{N}{2}$ for a non-negative integer $N$. Elements of $O$ are called open, elements of $C$ are called closed.
		
		A morphism $((O_1,O_2),(C_1,C_2),G)\to ((O'_1,O'_2),(C'_1,C'_2),G')$ is defined only for $G=G'$ and it is a a quadruple of bijections $O_1\stackrel{\sim}{\to}O'_1$, $O_2\stackrel{\sim}{\to}O'_2$, $C_1\stackrel{\sim}{\to}C'_1$ and $C_2\stackrel{\sim}{\to}C'_2$. 
	\end{definition}
	
	To define a $2$-colored properad, we replace in Definition \ref{DEFModOp} the category  $\DCor$ by $\DCor_2$, the characteristic $\chi$ is now $\chi = 2G + |O_1| + |O_2| + |C_1| + |C_2|-2 $, and also we consider only operations of the form
	\begin{align*}
	&\stackrel{(\eta_o,\eta_c)}{\ooo{(B_o,B_c)}{(A_o, A_c)}}:
	((O_1,O_2\sqcup B_o),(C_1,C_2\sqcup B_c),\chi_1)\tp ((O'_1\sqcup A_o,O'_2),(C'_1\sqcup A_c,C'_2),\chi_2)\\
	&\to ((O_1\sqcup O'_1,O_2\sqcup O'_2),(C_1\sqcup C'_1,C_2\sqcup C'_2),\chi),
	\end{align*} 
	for bijections $\eta_o:B_o\stackrel{\sim}{\to} A_o$ and $\eta_c:B_c\stackrel{\sim}{\to} A_c$.\footnote{Subscripts $o$ and $c$ again correspond to open and closed, respectively.}  The modification of axioms is obvious, we leave it to the reader to fill in the details.
	
	\begin{example}{The open-closed Frobenius properad $\mathcal{OCF}$}. For the $2$-colored properad $\mathcal{OCF}$, the degree $0$ vector space $\mathcal{OCF}((O_1,O_2),(C_1,C_2),G)$ is generated by homeomorphism classes of  2-dimensional compact oriented stable surfaces with genus $g$, $|O_1|$ open outputs and $|O_2|$ open inputs distributed over $b_1$ and $b_2$ open boundaries respectively and  $|C_1|$ closed outputs and $|C_2|$ closed inputs in the interior, $G = 2g +b +(|C_1| +|C_2|)/2-1$ with $b=b_1+b_2$. The $\Sigma$-action preserves the colors and the operations are defined by gluing open/closed inputs into open/closed outputs.
	\end{example}
	
	\begin{remark}
		So far, we discussed only linear properads, i.e., properads in the category of (differential, graded) vector spaces $\mathsf{Vect}$. It follows from the definitions that all our examples discussed so far are linearizations of properads in sets. For example, the (closed) Frobenius properad $\mathcal F$ is a linearization of the terminal $\mathsf{Set}$-properad. This can be compared to the modular operad $\mathtt{Mod}(Com)$, the modular envelope of the cyclic operad $Com$. This modular operad is a linearization of $\mathtt{Mod}(*_C)$\footnote{the modular envelope of the terminal cyclic operad $*_C$ in $\mathsf{Set}$}, the terminal modular operad in $\mathsf{Set}$ \cite{Markl_Modular envelopes}.
		In \cite{Markl_Modular envelopes}, Markl also formulates the following Terminality principle:
		\\
		\\
		{\it{For a large class of geometric objects there exists a version of modular operads such
				that the set of isomorphism classes of these objects is the terminal modular $\mathsf{Set}$-operad of a given
				type.}}
		\\
		\\
		It could be interesting to formulate a similar principle also in the world of properads.
		
	\end{remark}

	\subsection{Cobar complex}
	
	The cobar complex of a properad $\oP$ is a properad denoted by ${C\oP}$. It is the free properad generated by the suspended dual of $\oP$,  with the differential induced by the duals of structure maps.
	Roughly speaking, $C{\oP}$ is spanned by directed graphs with no directed circuits and its vertices are  decorated with elements of $\oP^{\#}$.

	\begin{definition} 
		A graph consists of vertices and half-edges.
		Exactly one end of every half-edge is attached to a vertex.
		The other end is either unattached (such an half-edge is called a leg) or attached to the end of another half-edge (in that case, these two half-edges form an edge). Every end is attached to at most one vertex/end. The half-edge structure for vertex $G_1$ of the graph $\gr{G}$ is indicated on the following picture on the left.
	\end{definition}
	\begin{definition}
		In a directed graph, every half-edge has assigned an orientation such that two half-edges composing one edge have the same orientation. The half-edges attached to each vertex are partitioned into incoming and outgoing half-edges.
		
		A directed circuit in such graph is a set of edges such that we can go along them following their orientation and get back to the point where we started.
	\end{definition}

	We require that to every vertex $V_i$ a nonnegative integer $G_i$ is assigned. We define $$G:=\dim_{\Q}H_1(\gr{G},\Q) + \sum_i G_i$$ to be the genus of the graph. The stable graphs then fulfill the condition
	$$\chi_i=2(G_i-1)+|C_i| + |D_i|>0 ,$$
	for every vertex $V_i$, where $|C_i|$ and $|D_i|$ denotes the number of outgoing resp. incoming half-edges attached to $V_i$. 
	
	Consider a finite directed graph $\gr{G}$ with no directed circuits and with integers $G_i$ assigned to each vertex as is indicated on the picture on the right. 
	
	\begin{center}
		\PICGraph
	\end{center}

	Finally, we require that the incoming legs of $\gr{G}$  are in bijection with the set $D$ and outgoing legs with $C$.\footnote{In \cite{GetzlerModop}, it is shown that the number of  isomorphism classes of (ordinary) stable graphs with legs labeled by the set $[n]$ and with the fixed genus $G$ is finite. The additional conditions on graphs, i.e., being directed with no directed circuits, will obviously not change this.}
	The graph $\gr{G}$ is ``decorated'' by an element $$(\uparrow\! V_1\wedge\cdots\wedge\uparrow\! V_n) \tp (P_1\tp \cdots \tp P_n),$$
	where $V_1,\ldots V_n$ are all vertices of $\gr{G}$, $\uparrow\! V_i$'s are formal elements of degree $+1$, $\wedge$ stands for the graded symmetric tensor product and $P_i\in\oP(C_i, D_i,\chi_i)^{\#}$, for every vertex $V_i$. Then the isomorphism class of $\gr{G}$ together with $(\uparrow\! V_1\wedge\cdots\wedge\uparrow\! V_n) \tp (P_1\tp \cdots \tp P_n)$ is an actual element of ${C\oP}(C,D,\chi)$.
	
	The operation $(\stackrel{\eta}{\ooo{B}{A}})_{{C\oP}}$ is defined by grafting of graphs, attaching together $|A|$ pairs of incoming and outgoing legs with the suitable orientation so that no directed circuits are formed.
	
	The differential $\partial_{{C\oP}}$ on ${C\oP}$ is the sum of the differential $d_{P^{\#}}$ and of the differential given by the dual of $(\stackrel{\eta}{\ooo{B}{A}})$ which adds one vertex $V$, $|A|$ edges attached to it and modifies the decoration of $\gr{G}$. 
	For an explicit formula, it is enough to consider a graph $\gr{G}$ with one vertex. On such a graph we have
	\begin{equation}\label{differencial}
	\partial_{{C\oP}}=  d_{P^{\#}}\otimes\id + \sum_{\substack{C_1 \sqcup C_2=C\\ D_1\sqcup D_2 = D \\ \chi=\chi(\chi_1,\chi_2, A, B,\eta)\\ \chi_1, \chi_2>0 }} \dfrac{1}{|A|!} (\stackrel{{(C_1,D_1\sqcup B, \chi_1)\ }{\eta}{\ (C_2\sqcup A, D_2,\chi_2)}}{\ooo{B}{A}})^{\#}_P \otimes (\uparrow V \wedge \cdot ), \end{equation}
	where 
	\begin{equation}\label{zobrazeni eta}
	(\stackrel{{(C_1,D_1\sqcup B, \chi_1)\ }{\eta}{\ (C_2\sqcup A, D_2,\chi_2)}}{\ooo{B}{A}})^{\#}_P:P(C,D,\chi)^{\#}\rightarrow P(C_1,D_1\sqcup B, \chi_1)^{\#} \otimes P(C_2\sqcup A, D_2,\chi_2)^{\#},\end{equation}
	for stable vertices $(C_1,D_1\sqcup B, \chi_1)$ and $(C_2\sqcup A, D_2,\chi_2)$. 
	For a general stable graph, the differential extends by the Leibniz rule. 
	
	\begin{remark}Here we should clarify the used notation. The sum is over pairs of sets $C_1, C_2$ and $D_1, D_2$ as indicated and also over characteristics $\chi_1, \chi_2$ and the bijection $\eta$\footnote{Notice, that by giving $\eta$ we also identify the sets $A$, $B$ and their size.} such that $\chi_1, \chi_2>0$ and the result of $\stackrel{{(C_1,D_1\sqcup B, \chi_1)\ }{\eta}{\ (C_2\sqcup A, D_2,\chi_2)}}{\ooo{B}{A}}$ gives a component of the given characteristic $\chi$. Such sum is obviously finite.
		
		For example, in the case of closed Frobenius properad where the characteristic is additive the sum is just over $G_1, G_2, \eta$ such that  $1\leq |A|\leq G+1, G_1+ G_2+|A|-1=G$ for a given $G$.
		
		We will use this shortened notation also in the following.
	\end{remark}

	{In the above formula we should make a choice of the ``new vertex'' $V$ out of the two vertices created by the splitting of the original one. 
		Since  we consider only connected directed graphs with no directed circuits, the new $\vert A\vert$ edges in the resulting graph will necessarily start in one vertex and end in the another one. We can choose any of them as the new one but once the choice is made, we have stick to it consistently when extending the differential using the Leibniz rule. The decoration by graded symmetric product of degree one elements then ensures that the $\partial_{{C\oP}}$ is really a differential.}
	
	To avoid problems with duals, we assume that the dg vector space $\oP(C,D,\chi)$ is finite dimensional for any triple $(C,D,\chi)$ whenever ${C\oP}$ appears. 
	This is sufficient for our applications, though it can probably be avoided using coproperads.

	Finally, let us, without going into details, mention the following: The cobar complex of a properad $\oP$ is in fact a double complex with the differentials being the two terms in the above formula (\ref{differencial}). Each component ${C\oP}(C,D,\chi)$ is given by a colimit of $(\bigwedge_{i=1}^{n}\uparrow V_i ) \tp \oP(C,D,\chi)^{\#}$ over all iso classes of directed connected graphs $\gr{G}$ with $n$ vertices with $|D|$ inputs and $|C|$ outputs.
	

	\subsection{\texorpdfstring{The endomorphism properad}{Endomorphism properad}}
	
	Let $(V,d)$ be a (dg) vector space. 
	
	\begin{definition}
		For any set $C$, $|C|=n$, we define the unordered product $\bigodot_{c\in C} V_c$ of the collection of vector spaces $\{V_c\}_{c\in C}$ as the vector space of equivalence classes of usual tensor products
		\begin{equation}
		\label{v_patek_domu_za_Jaruskou}
		v_{\omega(1)} \ot \cdots \ot v_{\omega(n)} \in V_{\omega(1)} 
		\ot \cdots \ot V_{\omega(n)},\ \omega : [n]
		\stackrel\cong\longrightarrow C,
		\end{equation}
		modulo the identifications
		\[
		v_{\omega(1)} \ot \cdots \ot v_{\omega(n)} 
		\sim \epsilon(\sigma)\
		v_{\omega\sigma(1)} \ot \cdots \ot v_{\omega\sigma(n)},\ \sigma \in \Sigma_n,
		\]
		where $\epsilon(\sigma)$ is the Koszul sign of the
		permutation $\sigma$.
	\end{definition}
	
	\begin{lemma}
		\label{bila_nemoc}
		Let
		$\sigma : C \to D$ be an isomorphism of finite sets, $\{V_c\}_{c\in C}$
		and $\{W_d\}_{d\in D}$ collections of graded vector spaces, $V_c = W_d =V$ for all $c \in C$, $d \in D$.
		Then the assignment
		\[
		\bigodot_{c\in C}V_c \ni \big[v_{\omega(1)}\tp \cdots \tp v_{\omega(n)}\big]
		\longmapsto \big[w_{\sigma\omega(1)}\tp \cdots \tp
		w_{\sigma\omega(n)}\big]
		\in \bigodot_{d \in D}V_d
		\]
		with
		$w_{\sigma\omega(i)} := v_{\omega(i)} \in
		V_{\sigma\omega(i)}$, $1 \leq i
		\leq n$,
		defines a natural map
		\[
		\overline\sigma :
		\bigodot_{c \in C}V_c \to \bigodot_{d \in D}V_d
		\]
		of unordered products
	\end{lemma}
	
	\begin{proof}
		A direct verification.
	\end{proof}
	
	\begin{lemma}
		\label{Tequila_v_lednici}
		For disjoint finite sets $C', C''$, one has a canonical isomorphism
		\[
		\bigodot_{c' \in C'} V_{c'} \tp \bigodot_{c'' \in C''}V_{c''} 
		\cong
		\bigodot_{c\in C'\sqcup\, C''} V_c.
		\]
	\end{lemma}
	
	\begin{proof}
		Each $\omega' : [n]  \stackrel\cong\to 
		C'$ and $\omega'' : [m] \stackrel\cong\to 
		C''$ determine an isomorphism
		\[
		\omega' \sqcup \omega'' : [n+m] \stackrel\cong\longrightarrow 
		C'\sqcup C''
		\]
		by the formula
		\[
		(\omega' \sqcup \omega'')(i) :=
		\begin{cases}
		\omega'(i),&\hbox{if $1 \leq i \leq n$, and}
		\cr
		\omega''(i-n),&\hbox{if $n < i \leq n+m$\ .}
		\end{cases}
		\]
		The isomorphism of the lemma is then given by the assignment
		\[
		[v_{\omega'(1)} \tp \cdots \tp v_{\omega'(n)}] \tp
		[v_{\omega''(1)} \tp \cdots \tp v_{\omega''(m)}]
		\mapsto [v_{(\omega' \sqcup\, \omega'')(1)} \tp \cdots \tp 
		v_{(\omega' \sqcup\, \omega'')(n+m)}].
		\]
	\end{proof}
	
	\begin{example}\label{unordIso}
		Let $C = \{c_1,\ldots,c_n\}$. By iterating Lemma~\ref{Tequila_v_lednici}
		one obtains a canonical isomorphism
		\[
		\bigodot_{c \in C} V_c \cong V_{c_1} \tp \cdots \tp V_{c_n}
		\]
		which, crucially, depends on the order of elements of $C$.
		In particular, for $C=[n]$, $V_c = V$, $c\in C$, we have an iso $\iota_\psi: V^{\otimes |n|} \to \bigodot_{[n]} V$ for every permutation $\psi$. In particular, we have the iso $\iota_n:=\iota_{1_{[n]}}$ corresponding to the natural ordering on the set $[n]$.
	\end{example}
	
	We are finally ready to define the endomorphism properad $\oEnd{V}$.
	
	\begin{definition}\label{Dnes_prednaska_v_Myluzach}
		For $(C,D)\in\DCor$, $\chi >0$ define $$\oEnd{V}(C,D,\chi) := \op{Hom}_{\K}(\bigodot_D V,\bigodot_C V).$$
		
		Let $\bar f\in \op{Hom}_{\K}(V_{d_1} \tp \cdots \tp V_{d_n},V_{c_1} \tp \cdots \tp V_{c_m})$ correspond to $f\in \op{Hom}_{\K}(\bigodot_D V,\bigodot_C V)$, under the above isomorphism in Example \ref{unordIso}. Then the differential on $\oEnd{V}$ is given, by abuse of notation, as
		\begin{equation}\label{diferencial endomorphismove}
		d(\bar f) = \sum_{i=0}^{m-1}(\id^{\tp i}\tp d\tp\id^{\tp m-i-1}) \bar f-(-1)^{\dg{\bar f}} \sum_{i=0}^{n-1} \bar f  (\id^{\tp i}\tp d\tp\id^{\tp n-i-1})
		\end{equation}
		
		Given a morphism $(\rho,\sigma) : (C,D) \to (C',D')$ in $\DCor$, define
		\begin{align*}
		\oEnd{V}(\rho,\sigma)  : \oEnd{V}(C,D,\chi) &\to \oEnd{V}(C',D',\chi )\\
		f&\mapsto\overline \rho \,f \, \overline \sigma,
		\end{align*}
		for $f\in \op{Hom}_{\K}(\bigodot_D V,\bigodot_C V)\in \oEnd{V}(C,D,\chi)$ and $\overline\rho,\,\overline\sigma$ as in~Lemma \ref{bila_nemoc}.
		
		For
		$f \in \oEnd{V}\big(C_2\sqcup A ,D_2, \chi_2\big)$ and $g \in \oEnd{V}\big(C_1, D_1 \sqcup B, \chi_1\big)$ let $$g\stackrel{\eta}{\ooo{B}{A}}f \in \oEnd{V}\big(C_1 \sqcup
		C_2, D_1,  \sqcup
		D_2, \chi\big)$$ be the composition
		\begin{align*}
		\bigodot_{d \in D_1 \sqcup\ D_2} V_{d} &\stackrel\cong\longrightarrow
		\bigodot_{d \in D_1} V_{d} \tp 
		\bigodot_{d' \in D_2} V_{d'}
		\stackrel{\id \tp f}{\longrightarrow}\bigodot_{d \in D_1} V_{d}  \tp \bigodot_{c \in C_2\sqcup A} V_{c}
		\\ 
		& 
		\stackrel\cong\longrightarrow\bigodot_{d \in D_1}V_{d}\tp \bigodot_{ a\in A} V_{a} \tp \bigodot_{c \in C_2} V_{c}
		\stackrel{1\tp \eta^{-1}\tp1}{\longrightarrow} \bigodot_{d \in D_1}V_{d}\tp \bigodot_{ b\in B} V_{b} \tp \bigodot_{c \in C_2} V_{c}
		\\
		&
		\stackrel\cong\longrightarrow  \bigodot_{d \in D_1 \sqcup B} V_{d}
		\tp
		\bigodot_{c \in C_2} V_{c} \stackrel{g \tp \id}{\longrightarrow} \bigodot_{c \in C_1} V_{c} \tp 
		\bigodot_{c' \in C_2} V_{c'}\stackrel\cong\longrightarrow \bigodot_{c \in C_1 \sqcup\ C_2} V_{c}
		\end{align*}
		in which the isomorphisms are easily identified with those of
		Lemma~\ref{Tequila_v_lednici}.
	\end{definition}
	We leave as an exercise to verify that the collection $$\oEnd{V} =
	\{\oEnd{V}(C,D, \chi)|(C,D) \in \DCor, \chi >0\}$$ with the above  operations is a properad.

	It is now straightforward to describe the skeletal version $\bar{\mathcal{E}}_V$ of the endomorphism properad $\mathcal{E}_V$ 
	$$\bar{\mathcal{E}}_V (m,n,\chi)= \op{Hom}_{\K}(V^{\tp n},
	V^{\tp m})\cong V^{\tp m}\otimes V^{\#{\tp n}},$$
	where the last isomorphism is explicitly for $v_1\tp\ldots\tp v_m \tp\alpha_1\tp\ldots \tp \alpha_ n \in V^{\tp m}\otimes V^{\#{\tp n}}$ 
	$$ v_1\tp\ldots\tp v_m \tp\alpha_1\tp\ldots \tp \alpha_ n: w_n\tp\ldots\tp w_1 \mapsto \alpha_1(w_1)\ldots \alpha_n(w_n)v_1\tp\ldots\tp v_m$$ 
	
	For $(\rho,\sigma)\in \Sigma_m\times \Sigma_n$, 
	$$(\rho,\sigma): v_1\tp\ldots\tp v_m\tp  \alpha_1\tp\ldots\tp \alpha_n\mapsto \pm\ v_{\rho^{-1}(1)}\tp\ldots\tp v_{\rho^{-1}(m)}\tp \alpha_{\sigma^(1)}\tp\ldots\tp \alpha_{\sigma(n)},$$
	where $\pm$ is the product of the respective Koszul signs corresponding to permutations $\rho$ and $\sigma$.
	
	The differential $d$ is given by the natural  extension of $d$ on $V$, as a degree one derivation, to $  V^{\tp m}\otimes V^{\# \tp n}$.\footnote{Recall, $(d\alpha)(v)= (-1)^{|\alpha|}\alpha(dv)$.}
	
	Finally, the operations $\stackrel{\xi}{\bar{\ooo{N}{M}}}$ are described as follows. Let $N = n_1+ [|N|]\subset [n_1 +|N|]$, $M = [|N|]\subset [m_2+|N|]$ and $\xi(n_1+|N|-i+1)=i$
	then $\stackrel{\xi}{\bar{\ooo{N}{M}}}$ is defined by the following assignment:
	\begin{align*}
	\stackrel{\xi}{\bar{\ooo{N}{M}}}:\,\,&(v_1\tp \ldots\tp v_{m_1}\tp\alpha_1\tp\ldots\tp \alpha_{n_1+|N|})
	\tp(w_1\ldots\tp w_{m_2+|N|}\tp\beta_1\tp\ldots\tp \beta_{n_2}  )\\ 
	& \hspace{-2.7em} \mapsto \pm \Pi_{i=1}^{|N|} \alpha_{n_1+|N|-i+1}(w_{i}) v_1\!\tp\!\ldots\! \tp v_{m_1}\!\tp w_{|N| +1}\!\tp\! \ldots \!\tp w_{m_2+|N|}\!\tp\! \alpha_1\!\tp\!\ldots\!\tp \alpha_{n_1}\!\tp \!\beta_{1}\!\tp\!\ldots \!\tp\! \beta_{n_2},
	\end{align*}
	where $\pm$ is the obvious Koszul sign, coming from commuting consecutively $w_{|N|+1}, \ldots, w_{m_2+|N|}$ trough $\alpha_{n_1}\otimes \ldots\tp \alpha_1 $.
	The general case is then easily determined by the equivariance of the operations $\stackrel{\xi}{\bar{\ooo{N}{M}}}$.
	
	\begin{remark} The above introduced skeletal version of the endomorphism properad is equivalent to the one which uses unordered tensor products  
		$\otimes_{[n]}V$ instead of ordinary ones $V^{\otimes n}$. This is possible due to  Example \ref{unordIso} according to which we have the canonical isomorphism $\otimes_{[n]}V\cong V^{\otimes n}$ corresponding to the natural ordering on $[n]$.
	\end{remark}
	
	Finally, we briefly discuss the $2$-colored version of the endomorphism properad.
	\begin{definition}
		Let $V_{\uo}\oplus V_{\uc}$ be an abbreviation for the direct sum of dg vector spaces $(V_{\uo},d_{\uo})$ and $(V_{\uc},d_{\uc})$.
		Let 
		$$\oEnd{V_{\uo}\oplus V_{\uc}}((O_1,O_2),(C_1, C_2),\chi) := \op{Hom}_{\K}(\bigodot_{O_2} V_{\uo} \tp \bigodot_{C_2} V_{\uc},\bigodot_{O_1} V_{\uo} \tp \bigodot_{C_1} V_{\uc}).$$
		
		The $\Sigma$-action and the operations are defined analogously to the $1$-colored case.  
		
		The notion of cobar complex of a $2$-colored properad is defined using a suitable definition of $2$-colored directed graphs.
		We leave it to the reader to fill in the details.
	\end{definition}
	\subsection{\texorpdfstring{Algebra over a properad}{Algebra over a properad}} \label{SECTAlgOverTwisted}
	
	\begin{definition} \label{DEFAlgOverTwisted}
		Let $\oP$ be a properad.
		An algebra over $\oP$ on a dg vector space $V$ is a properad morphism $$\alpha : \oP \to \oEnd{V},$$
		i.e. it is a collection of dg vector space morphisms 
		$$\{\alpha(C,D,\chi):\oP(C,D,\chi) \to\oEnd{V}(C,D,\chi)\ | \ (C,D)\in\DCor, \chi>0\}$$ 
		such that (in the sequel, we drop the notation $(C,D,\chi)$ at $\alpha(C,D, \chi)$, for brevity)
		\begin{enumerate}
			\item $\alpha \comp \oP(\rho,\sigma) = \oEnd{V}(\rho, \sigma) \comp \alpha$ for any morphism $(\rho, \sigma)$ in $\DCor$
			\item $\alpha \comp ({\stackrel{\eta}{\ooo{B}{A}}})_{\oP} = ({\stackrel{\eta}{\ooo{B}{A}}})_{\oEnd{V}} \comp (\alpha\tp\alpha)$
			
			Algebra over a $2$-colored properad is again defined by replacing $\DCor$ by $\DCor_2$. 
		\end{enumerate}
	\end{definition}
	
	In practice, however, one is rather interested in skeletal version of $\alpha$'s, i.e., $\Sigma_m\times \Sigma_n$-equivariant maps 
	$$\alpha(m,n,\chi): \bar{\oP}(m,n,\chi)\to \bar{\mathcal{E}}_V(m,n,\chi)$$ intertwining between the respective $\stackrel{\xi}{\bar{\ooo{N}{M}}}$ 
	operations.

	\begin{remark}
		Note that the above formula 2. is compatible with any composition law for the degree $G$, or equivalently for the Euler characteristic $\chi$. This is because, for fixed values of $m$ and $n$, the vector spaces $\mathcal{E}_V(m,n,\chi)$ are independent of the actual value of $\chi$.  So we always can  choose the composition law for $\chi$ in the endomorphism properad $\mathcal{E}_V$ so that it respects the one for $\mathcal{\oP}$.
	\end{remark}
	
	\subsection{\texorpdfstring{Algebra over the cobar complex}{Algebra over the Feynman transform}}
	
	The following theorem is essentially the only thing we need from the theory of the cobar transform.
	Compare to Feynman transform for modular operads \cite{BarannikovModopBV}. 
	
	In order to describe an algebra over the cobar complex, it is enough to consider graphs with one vertex.
	
	\begin{theorem} \label{LEMMAAlgOverFeynTrans}
		An algebra over the cobar complex $C{\oP}$ of a properad $\oP$ on a dg vector space $V$ is uniquely determined by a collection of degree $1$ linear maps
		$$\set{\alpha(C,D,\chi):\oP(C,D,\chi)^{\#}\to\oEnd{V}(C,D,\chi)}{(C,D)\in\DCor, \chi>0 },$$
		(no compatibility with differential on $\oP(C,D,\chi)^{\#}$!) such that 
		$$
		\oEnd{V}(\rho, \sigma) \comp \alpha(C,D,\chi) = \alpha(C',D',\chi) \comp \oP(\rho^{-1},\sigma^{-1})^{\#} \label{EQAlfOverFTOne}$$
		for any pair of bijections $(\rho,\sigma):(C,D)\xrightarrow{\sim} (C',D')$  and 
		\begin{align}
		d \comp \alpha(C,D,\chi) &=\alpha(C,D,\chi) \comp d_{\oP^{\#}}   + \nonumber \\
		&\hspace{-5em}     +  \sum_{\substack{C_1\sqcup C_2 = C \\D_1\sqcup D_2 = D\\ \chi=\chi(\chi_1,\chi_2,A,B,\eta)\\\chi_1,\chi_2>0\\}} \dfrac{1}{|A|!}(\stackrel{\eta}{\ooo{B}{A}})_{\oEnd{V}} \comp \left( \alpha(C_1, D_1\sqcup B,\chi_1)\tp\alpha(C_2\sqcup A,D_2,\chi_2) \right) \comp (\stackrel{\eta}{\ooo{B}{A}})^{\#}_{\oP}\label{EQAlfOverFTTwo},
		\end{align}
		where $(\stackrel{\eta}{\ooo{B}{A}})^{\#}_{\oP}$ is a shorthand notation for $(\stackrel{{(C_1,D_1\sqcup B, \chi_1)\ }{\eta}{\ (C_2\sqcup A, D_2,\chi_2)}}{\ooo{B}{A}})^{\#}_P$ from (\ref{zobrazeni eta})
		$$(\stackrel{\eta}{\ooo{B}{A}})^{\#}_{\oP} : \oP(C, D,\chi)^{\#} \to \oP(C_1,D_1 \sqcup B,\chi_1)^{\#} \tp \oP(C_2, D_2 \sqcup A,\chi_2)^{\#}.$$ 
	\end{theorem}
	
	It will also be useful to have the skeletal version of the above theorem.
	
	\begin{lemma} \label{LEMMAAlgOverFeynTrans1}
		Algebra over the cobar complex $C{\oP}$ of a properad $\oP$ on a dg vector space $V$ is uniquely determined by a collection
		$$\set{\bar{\alpha}(m,n,\chi):\bar{\oP}(m,n,\chi)^{\#}\to\bar{\mathcal E}_{V}(m,n,\chi)}{([n],[m])\in\DCor}$$
		of degree $1$ linear maps (no compatibility with differential on $\bar{\oP}(m,n,\chi)^{\#}$!) such that\footnote{In the sequel, we simplify the notation a bit further: the $(m,n,\chi)$ at ${\bar{\alpha}}(m,n,\chi)$ is usually omitted and so is the symbol $\circ$ for composition of maps.}
		$$
		\bar{\mathcal E}_{V}(\rho, \sigma)\bar{\alpha} = \bar{\alpha} \bar{\oP}(\rho^{-1},\sigma^{-1})^{\#} \label{EQAlfOverFTOne}$$
		for any pair $(\rho,\sigma)\in \Sigma_m \times \Sigma_n$  and 
		\begin{align}
		d  \bar{\alpha} &= \bar{\alpha} d_{\bar{\oP}^{\#}}   + \nonumber \\
		&\hspace{-1.5em} + \!\!\!\!\!\!\!\!\!\!\!\!\!\!\!\! \sum_{\substack{C_1\sqcup C_2 = [m] \\D_1\sqcup D_2 = [n] \\  \,\,\,\,\,\,\,\chi=\chi(\chi_1,\chi_2,A, B,\eta)\\\chi_1,\chi_2>0}} \!\!\!\!\!\!\!\!\!\!\!\!\!\!\!   \oEnd{V}(\kappa_1\!\sqcup\kappa_2\rho_A^{-1}, \lambda_1\rho_B^{-1}\sqcup\lambda_2) (\!\stackrel{\kappa_2^{-1}\eta\lambda_1}{\bar{\ooo{\lambda_1^{-1}(B)}{\kappa_2^{-1}(A)}}})_{\oEnd{V}} \left(\bar{\alpha}\tp\bar{\alpha}\right) \left(\oP(\kappa_1,\lambda_1)^{\#}\!\otimes \! \oP(\kappa_2,\lambda_2)^{\#}\right) (\!\stackrel{\eta}{\ooo{B}{A}})^{\#}_{\bar{\oP}} \label{EQAlfOverFTTwo},
		\end{align}
		where $\kappa_1: [|C_1|] \xrightarrow{\sim} C_1, \lambda_1: [|D_1|+|B|] \xrightarrow{\sim} D_1\sqcup B, \kappa_2: [|C_2|+|A|] \xrightarrow{\sim} C_2\sqcup A, \lambda_2: [|D_2|]\xrightarrow{\sim} D_2$, $\rho_B: [|D_1|+|B|] -B \xrightarrow{\sim} |D_2|+[|D_1|],  \rho_A: [|C_2|+|A|] -A \xrightarrow{\sim} |C_1|+[|C_2|]$ 
		are arbitrary bijections.
	\end{lemma}
	\begin{remark}
		The above discussion straightforwardly carries over to the $2$-colored case, the reader can easily fill in the details.
	\end{remark}
	\subsection{Barannikov's type theory} \label{SSECBarannikov}
	{In Theorem 1 of \cite{BarannikovModopBV}, Barannikov observed that an algebra over the Feynman transform of modular operad $\mathcal{P}$ is equivalently
		described as a solution of a certain master equation in an algebra succinctly defined in terms of $\mathcal{P}$, cf. also Theorem 20 in \cite{DJM}. 
		In this section, we formulate the corresponding theorem for properads in our formalism and then adapt it to our applications.} 
	
	Assume $C_1,D_1,C_2,D_2, \kappa_1, \lambda_1, \kappa_2, \lambda_2$ are given as in lemma \ref{LEMMAAlgOverFeynTrans1}.
	
	\begin{definition}
		For a properad $\oP$, define 
		\begin{gather*}
		P(m,n,\chi):=\,^{\Sigma_m}\left(\oP([m],[n],\chi)\tp\oEnd{V}([m],[n],\chi)\right)^{\Sigma_n} \\
		P:= \prod_{\substack{n\geq 0,m\geq 0\\ \chi> 0}} P(m,n, \chi)
		\end{gather*}
		with $P(m,n,\chi)$ being the space of invariants under the diagonal $\Sigma_m \times \Sigma_n$ action on the tensor product.
		Let $P$ be equipped with a differential, given for $f\in P(m,n,\chi)$, by
		\begin{equation}\label{diferencial v baranikovi}
		d(f) := \left(d_{\oP([m],[n],\chi)}\tp\id_{\oEnd{V}([m],[n],\chi)}-\id_{\oP([m],[n],\chi)}\tp d_{\oEnd{V}([m],[n],\chi)} \right) (f),
		\end{equation}
		and a composition $\circ$ described as follows: Assume $ g\in P(m_1,n_1 +|B|,\chi_1),\ h\in P(m_2+|A|,n_2,\chi_2)$ and $|A|=|B|$, then the $(m=m_1+m_2, n=n_1+n_2, \chi=\chi(\chi_1, \chi_2, A, B, \eta))$ component of the composition $g\circ h$ is given by
		$$
		\sum\left((\stackrel{\eta}{\ooo{B}{A}})_{\oP} \tp (\stackrel{\eta}{\ooo{B}{A}})_{\oEnd{V}}\right) \sigma_{23}(\oP(\kappa_1,\lambda_1)\tp\oEnd{V}(\kappa_1,\lambda_1)\tp\oP(\kappa_2,\lambda_2)\tp\oEnd{V}(\kappa_2,\lambda_2)) (g\tp h).$$
		The differential and composition are extended by infinite linearity to the whole $P$.
		Here the sum is over $C_1\sqcup C_2=[m], D_1\sqcup D_2=[n],|C_1|=m_1, |C_2|=m_2, |D_1|=n_1, |D_2|=n_2$ and $\sigma_{23}$ is the flip excahning the two middle factors.
		Recall that $\kappa_1,\kappa_2,\lambda_1,\lambda_2$ depend on $C_1,C_2,D_1,D_2$.
	\end{definition}
	\begin{remark}
		Since the above definition of the composition $\circ$ doesn't depend on the choice of maps $\kappa_1,\kappa_2,\lambda_1,\lambda_2$ it might be sometimes useful to make a convenient choice of these. Without loss of generality we can assume $A \subset [m_2 + |A|]$ and $B \subset [n_1 + |B|]$ and hence relabel them as $M$ and $N$ respectively, just to follow our conventions from remark \ref{poznamka s unshuffly}. Let $\kappa_1$, $\lambda_2$ be increasing as well as $\lambda_1$ when restricted to $[n_1 + |N|] -N$ and $\kappa_2$ when restricted to $[m_2 + |M|] -M$. Then the $(m=m_1+m_2, n=n_1+n_2,\chi(\chi_1,\chi_2,M,N,\xi))$ component of the above composition $g\comp h$ can be rewritten as
		\begin{equation}\label{kompozice1}
		\sum(\oP(\rho,\sigma)\tp\oEnd{V}(\rho,\sigma))\left((\stackrel{\xi}{\bar{\ooo{N}{M}}})_{\oP} \tp (\stackrel{\xi}{\bar{\ooo{N}{M}}})_{\oEnd{V}}\right)   \sigma_{23}(g\tp h),
		\end{equation}
		with the sum running over all  $(m_1, m_2)$-shuffles $\rho$ and $(n_2, n_1)$-shuffles $\sigma$.
	\end{remark}
	\begin{theorem} \label{THMFeynBaran}
		Algebra over the cobar complex $C{\oP}$ on a dg vector space $V$ is equivalently given by a degree $1$ element
		$$L\in P:=\prod_{\substack{m,n\\ \chi > 0}}  \,^{\Sigma_m}\left(\oP([m],[n],\chi)\tp\oEnd{V}([m],[n],\chi)\right)^{\Sigma_n}$$
		satisfying the master equation
		\begin{equation}\label{MasterEquation}
		d(L)+L\comp L=0.
		\end{equation}
	\end{theorem}
	
	\begin{proof}[Sketch of proof]
		Consider the iso
		\begin{align}
		\op{Hom}_{\Sigma_C \times \Sigma_D}(\oP(C,D,\chi)^{\#},\oEnd{V}(C,D,\chi)) & \xrightarrow{\cong} \,^{\Sigma_C}\left( \oP(C,D,\chi)\tp\oEnd{V}(C,D,\chi) \right)^{\Sigma_D} \label{EQBigXiIso} \\
		\alpha & \mapsto \sum_{i} p_i\tp\alpha(p_i^{\#}) \no
		\end{align}
		where $\{p_i\}$ is a $\K$-basis of $\oP(C,D,\chi)$ and $\{p_i^{\#}\}$ is its dual basis.
		Under this iso, \eqref{EQAlfOverFTTwo} becomes the $(\chi, m, n)$-component of the master equation of this theorem.
		
	\end{proof}

	By \eqref{EQBigXiIso}, any $L\in P$ can be written in the form 
	$$L=\sum_{n,m,\chi} L_{m,n,\chi} =\sum_{n,m,\chi}\sum_{i}p_i\tp\alpha(p_i^{\#})$$
	for some collection $\alpha$ of $\Sigma_m \times \Sigma_n$-equivariant  maps  of degree 1
	$$\alpha([m],[n],\chi):\oP([m],[n],\chi)^\#\to \mathcal{E}_V([m],[n],\chi).$$
	
	Let $p_i$ be basis of $\oP([m],[n],\chi)$ and $p^\#_i$ the dual one. Put $f_{p_i} := \bar\alpha(p^\#_i): V^{\otimes n} \to V^{\otimes m}$.
	Also, pick a homogeneous basis $\{a_i\}$ of $V$ and denote $f_{p_iI}^J$ the respective coordinates of $f_{p_i}$, where $I:=(i_1,\ldots,i_{n})$ and $J:=(j_1,\ldots,j_{m})$ are multi-indices in $[\dim V]^{\times n}$ and in $[\dim V]^{\times m}$, respectively.
	
	Hence, we have an iso $Y$:
	\begin{align}
	Y : \,^{\Sigma_m}\left(\oP([m],[n]\chi)\tp\oEnd{V}([m],[n],\chi)\right)^{\Sigma_n} &\cong \oP([m],[n],\chi)\,_{\Sigma_m}\tp_{\Sigma_n} V^{\tp m}\tp ((V^{\#})^{\tp n})\label{EQFirstIdentification} \\
	\sum_i p_i\tp\alpha(p_i^{\#}) &\mapsto \frac{1}{m!n!}\sum_{i,I,J}  f_{p_iI}^J (p_i\,_{\Sigma_m}\tp_{\Sigma_n}(a_J\otimes \phi^I ))\no  
	\end{align}
	and the RHS is the space of coinvariants with respect to the diagonal $\Sigma_n\times \Sigma_m$ action on the tensor product. Here, $\{\phi^i\}$ is the basis dual to $\{a_i\}$. The coefficient $\frac{1}{n!m!}$ is purely conventional.
	In particular, we have 
	\begin{equation}\label{Lcoinv}
	L=\sum_{n,m,\chi} \frac{1}{m!n!}\sum_{i,I,J}  f_{p_iI}^J (p_i\,_{\Sigma_m}\tp_{\Sigma_n}(a_J\otimes \phi^I ))
	\end{equation}
	
	The obvious inverse $Y^{-1}$ is
	$$Y^{-1}:p\,_{\Sigma_m}\tp_{\Sigma_n} (a_J\tp\phi^I)\mapsto \sum_{(\rho,\sigma) \in \Sigma_m \tp \Sigma_n} \oP(\rho,\sigma)(p)\tp\oEnd{V}(\rho,\sigma)(a_J\otimes \phi^I). $$
	Denote
	\begin{gather*}
	\tilde{P}(m,n,\chi) := \left(\oP([m],[n],\chi)\,_{\Sigma_m}\tp_{\Sigma_n}(V^{\tp m}\tp (V^{\#})^{\tp n}\right)\\
	\tilde{P} := \prod_{m,n,\chi}\tilde{P}(m,n,\chi).
	\end{gather*}
	Then $P\cong\tilde{P}$ and we can transfer the operations $d$ and $\circ$ from $P$ to $\tilde {P}$. 
	We start with the differential $\tilde{d}$ on $\tilde{P}$, which is obvious
	\begin{gather}
	\tilde{d} \left(p\,_{\Sigma_{m}}\otimes_{\Sigma_{n}} (a_J\tp\phi^I)\right)= d_\oP (p)\,_{\Sigma_{m}}\otimes_{\Sigma_{n}}(a_{J}\otimes \phi^{I})
	-(-1)^{|p|} p\,_{\Sigma_{m}}\otimes_{\Sigma_{n}} d_{\mathcal{E}_V}(a_J\tp\phi^I)
	\label{EQBVDiff}
	\end{gather}
	Concerning  the composition $\tilde{\circ}$, this is a bit more complicated, but also straightforward.
	
	\begin{equation*}
	\begin{tikzpicture} [baseline=-\the\dimexpr\fontdimen22\textfont2\relax]
	\matrix (m) [matrix of math nodes, row sep=3em, column sep=2.5em, text height=1.5ex, text depth=0.25ex]
	{ P\tp P & & \tilde{P} \tp \tilde{P} \\
		P & & \tilde{P} \\ };
	\path[->,font=\scriptsize] (m-1-1) edge node[below]{$\cong$} node[above]{$Y\tp Y$} (m-1-3);
	\path[->,font=\scriptsize] (m-1-1) edge node[left]{$\circ$} (m-2-1);
	\path[->,dashed,font=\scriptsize] (m-1-3) edge node[right]{$\tilde{\circ}$} (m-2-3);
	\path[->,font=\scriptsize] (m-2-1) edge node[above]{$Y$} node[below]{$\cong$} (m-2-3);
	\end{tikzpicture}
	\end{equation*}
	Chasing the above commutative diagram, we obtain:
	\begin{align}\label{skladani v baranikovi}
	&\left(p_1\,_{\Sigma_{m_1}}\otimes_{\Sigma_{n_1}} (a_{J_1}\otimes\phi^{I_1} )\right)\tilde{\circ}   \left(p_2\,_{\Sigma_{m_2}}\otimes_{\Sigma_{n_2}}(a_{J_2}\otimes \phi^{I_2})\right)=
	\\&= \sum_{M,N,\xi} \left((\stackrel{\xi}{\bar{\ooo{N}{M}}})_{\oP} (p_1 \tp p_2)\right)\,_{\Sigma_{m_1+m_2-|M|}}\otimes_{\Sigma_{n_1+n_2-|M|}} \left((\stackrel{\xi}{\bar{\ooo{N}{M}}})_{\oEnd{V}} (a_{J_1} \otimes \phi^{I_1} )\otimes (a_{J_2}\otimes \phi^{I_2})\right), \nonumber  
	\end{align}
	where the sum runs over all pairs of nonempty subsets $M\subset [m_2]$, $N\subset [n_1]$ with $|M|=|N|\leq\mbox{min}\{m_2,n_1\}$ and all isomorphisms $\xi$ between $N$ and $M$.
	\begin{remark} It is a straightforward check left to the reader to show that $(\tilde{P},\tilde{\circ}) $ forms a Lie-admissible algebra.
	\end{remark}

	\begin{quote}Let's assume that $\oP$ is a linear span of a properad in sets.\end{quote}
	That is, we assume that for each $([m],[n])\in\DCor$ there is a basis $\{p_i\}$ of $\oP([m],[n],\chi)$ which is preserved by the $\Sigma_m\times \Sigma_n$-action and the operations $\stackrel{\eta}{\ooo{B}{A}}$. This is obviously satisfied, e.g., for the closed Frobenius properad considered in this paper.
	With these choices, the coordinates $f_{p_iI}^J$ have the following simple
	invariance property
	$$f_{p_iI}^J = \pm f_{\oP(\rho,\sigma)(p_i) \sigma^{-1}(I)}^{\rho(J)},$$
	where $\pm$ is product of respective Koszul signs corresponding to $\rho(J)$ and $\sigma(I)$.
	
	We can decompose $\{p_i\}$ into $\Sigma_m\times \Sigma_n$-orbits indexed by $r$ and choose a representative $p_r$ for each $r$.
	Denote $O(p_r):=\Sigma_m \times\Sigma_m/\mathrm{Stab}(p_r)$ and also fix a section $\Sigma_m\times\Sigma_n/\mathrm{Stab}(p_r) \hookrightarrow \Sigma_m\times\Sigma_n$ of the natural projection, thus viewing $O(p_r)$ as a subset of $\Sigma_m\times\Sigma_n$.
	Hence the orbit of $p_r$ in $\oP([m],[n],\chi)$ is $\{\oP(\rho,\sigma)p_r \ |\ (\rho,\sigma)\in O(p_r)\}$ and it has $|O(p_r)|=\frac{n!m!}{|\mathrm{Stab}(p_r)|}$ elements.
	Hence, we can get an expression for elements of $\tilde{P}$ involving $p_r$'s only:

	\begin{gather}\label{EQAsInTale} 
	\frac{1}{m!n!} \sum_{i,I,J} \ f_{p_iI}^J (p_i\,_{\Sigma_m}\tp_{\Sigma_n}(a_J \otimes \phi^I)) = \sum_{r} \frac{1}{|\mathrm{Stab}(p_r)|} \sum_{I,J} \ f_{p_rI}^J (p_r\,_{\Sigma_m}\tp_{\Sigma_n}(a_J\otimes \phi^I))
	\end{gather}
	Thus the generating operator $L\in\tilde{P}$ can be expressed as
	\begin{gather} \label{EQGenFctionTilde}
	L=\sum_{m,n,\chi}\sum_{r,I,J} \frac{1}{|\mathrm{Stab}(p_r)|} f_{p_rI}^J \ (p_r\,_{\Sigma_m}\tp_{\Sigma_n}( a_J\otimes \phi^I)).
	\end{gather}
	
	Finally, it can be useful to have the following interpretation of the operation $\tilde{\circ}$. Here we shall assume the our corollas have always at least one input and one output, i.e. we assume $\mathcal{P}(C,D,\chi)$ to be nontrivial only if both $C$ and $D$ are non-empty and $m+n>2$,  for $G=0$. In this case, we  introduce, similarly to \cite{DJM}, positional derivations  
	\begin{equation}\label{pozicni derivace}
	\dfrac{\partial^{(k)}}{\partial a_j}\left(a_{i_1} \otimes \ldots \otimes a_{i_{m}}\right) =\left(-1\right)^{|a_j|(|a_{i_1}|+\ldots |a_{i_{k-1}}|)} \delta_j^{i_k} \left(a_{i_1} \otimes \ldots \otimes\widehat{a_{i_k}}\otimes \ldots \otimes a_{i_{m}}\right) \end{equation}
	and for sets $J=\lbrace j_1, \ldots j_{|N|}\rbrace$ and $K=\lbrace k_1, \ldots k_{|N|}\rbrace$
	$$\dfrac{\partial^{(K)}}{\partial a_J} = \dfrac{\partial^{(k_1)}}{\partial a_{j_1}} \ldots \dfrac{\partial^{(k_{|N|})}}{\partial a_{j_{|N|}}}.
	$$
	{Although the formula defining the positional derivative might seem obscure at the first sight, its usefulness will be obvious from the forthcoming formula (\ref{forthcoming}). The meaning of the positional derivative $\dfrac{\partial^{(k)}}{\partial a_j}$ is simple. Applied to a tensor product like $a_{i_1}\otimes \ldots \otimes a_{i_{m}}$ it is zero unless there is a tensor factor $a_j$ at the $k$-th position, in which case it cancels this factor and produces the relevant Koszul sign. We have introduced it because, in contrary to the left derivative familiar from the supersymmetry literature, here we do not have a rule how to commute the tensor factor $a_j$ to the left.}
	The ``inputs" from $(V^{\#})^{\tp n_1}$ in equation (\ref{skladani v baranikovi}) can then be interpreted as the partial derivations acting on the ``outputs" from $V^{\tp m_2}$, and hence we can interpret  elements of $\tilde P =\prod_{m,n,\chi} \tilde P(m,n,\chi)$ as differential operators acting on $\tilde P_+ := \prod_{k} \tilde P(k,0,\chi)$ as
	\begin{align}\label{forthcoming}
	& p_1\,_{\Sigma_{m_1}}\otimes_{\Sigma_{n_1}} (a_{J_1}\otimes \phi^{I_1}) : \\& p_2\,_{\Sigma_{m_2}}\otimes a_{J_2}\mapsto
	\pm \sum_{M,N,\xi}  \dfrac{\partial^{\xi(N)}}{\partial a_{N}} \left( a_M\right) (\stackrel{\xi}{\bar{\ooo{N}{M}}})_{\oP} (p_1 \tp p_2)\,_{\Sigma_{m_1+m_2-|M|}}\otimes_{\Sigma_{n_1-|M|}}
	a_{J_1} a_{J_2-M}, \nonumber 
	\end{align}
	where the sign $\pm$ is given as in (\ref{pozicni derivace}). Hence, in the master equation  $\tilde{d}\tilde{L}+\tilde{L}\tilde{\circ}\tilde{L}=0$ where $\tilde{L}=Y(L)$ with $Y$ being the iso (\ref{EQFirstIdentification}), the operation $\tilde \circ$ becomes the composition of differential operators. For this, recall that $\tilde{L}$ is of degree 1 so we can write $\tilde{L}\tilde{\circ}\tilde{L}=\frac{1}{2}[\tilde{L}\stackrel{\tilde{\circ}}{,}\tilde{L}]$ as the graded commutator.
	\section{\texorpdfstring{$IBL_\infty$}--algebras and their cousins} \label{SECTQC}
	
	\subsection{\texorpdfstring{$IBL_\infty$-algebras}{Loop homotopy algebras}} \label{SECTLoopHomAlg}
	\medskip
	The following statement appeared in \cite{Drummond}, cf. also \cite{Cieliebak ibl}.
	\begin{theorem}\label{THMMarklLoopViaOperads}
		The algebras over  the cobar complex  $\mathcal{CF}$ of the (closed) Frobenius properad are $IBL_\infty$-algebras.\footnote{ Note that our conventions are slightly different. Usually, for $IBL$-algebras one assumes that $n\geq 1$, $m\geq 1$, $G\geq 0$. We will comment on this later.} 
	\end{theorem}
	To prove it, recall the definition of the Frobenius properad $\mathcal{F}$ from example \ref{closed frob properad}. Each stable $\bar{\mathcal{F}}(m,n,\chi)$ is a trivial $\Sigma_m\times \Sigma_n$-bimodule spanned on one generator $p_{m,n,\chi}$. Hence,
	$$\bar{\mathcal{F}}(m,n,\chi) \,_{\Sigma_m}\tp_{\Sigma_n} (V^{\tp m}\tp V^{\#\tp n})\cong S^m(V) \tp S^n(V^\#)$$
	is the tensor product of the respective symmetric powers.
	It follows that formula (\ref{EQGenFctionTilde}) for the generating operator $L \in \tilde P$ is simplified to the form\footnote{Note, the invariance property $f^{\chi,J}_I= \pm  f^{\chi,\rho(J)}_{\sigma^{-1}(I)}$, $\pm$ being the product of Koszul signs corresponding to permutations $\rho$ and $\sigma$.}
	$$ L= \sum_{m,n,\chi} \sum_{I,J} \frac{1}{m!n!}f^{\chi,J}_I (a_J\tp\phi^I)$$
	where $f^{\chi,J}_I =(\bar{\alpha}(p^\#_{m,n,\chi}))^J_I$.
	
	
	Further, the algebra over cobar complex $\mathcal{CF}$ is given by (\ref{EQAlfOverFTTwo}). The differential $d_{\oP^{\#}}$ is for Frobenius properad trivial and the differential on $\oEnd{V}$ is given by (\ref{diferencial endomorphismove}). What is left is to exhibit the second term of the RHS of (\ref{EQAlfOverFTTwo}).
	\\
	As in formula (\ref{EQAlfOverFTTwo}), assume $A \subset [m_2 + |A|]$, $B \subset [n_1 + |B|]$, relabel them as $M$ and $N$ respectively, and assume $N=\{n_1+1,\ldots, n_1+|N|\}$, $M=\{1,\ldots, |N|\}$, $\xi(n_1 +k)=k$. The second term of the RHS in  (\ref{EQAlfOverFTTwo}) evaluated on the generator $p_{m,n,\chi}$ gives
	$$ 
	\sum_{\substack{m_1+m_2=m\\ n_1+ n_2=n}}\sum_{|N|=1}^{\frac{1}{2}(\chi-m-n)+2}\sum_{\chi_1} \sum_{\rho, \sigma}
	\rho\ (\stackrel{\xi}{\bar{\ooo{N}{M}}})_{\oEnd{V}} \left(\alpha_{m_1,n_1,\chi_1}\otimes \alpha_{m_2,n_2,\chi_2}\right)\sigma^{-1},
	$$
	where $\mathrm{max}\lbrace m_1+n_1+|N|-2, 1 \rbrace \leq \chi_1 \leq \mathrm{min}\lbrace\chi-m_2-n_2-|N|+2, \chi-1\rbrace $ by stability condition, $\alpha_{m_1,n_1,\chi_1}:= \bar{\alpha}(p^\#_{m_1,n_1,\chi_1})$, $\alpha_{m_2,n_2,\chi_2}:= \bar{\alpha}(p^\#_{m_2,n_2,\chi_2})$ and the last sum runs over shuffles $\rho,\sigma$ of type $(m_1,m_2)$ and $(n_2,n_1)$, respectively. If we denote the differential $d$ of dg vector space as $\alpha_{1,1,0}$ then together we get
	$$ 0=
	\sum_{\substack{m_1+m_2=m\\ n_1+ n_2=n}}\sum_{|N|=1}^{\frac{1}{2}(\chi-m-n)+2}\sum_{\chi_1=m_1+n_1+|N|-2}^{\chi-m_2-n_2-|N|+2} \sum_{\rho, \sigma}
	\rho\ (\stackrel{\xi}{\bar{\ooo{N}{M}}})_{\oEnd{V}} \left(\alpha_{m_1,n_1,\chi_1}\otimes \alpha_{m_2,n_2,\chi_2}\right)\sigma^{-1},
	$$
	which is, up to conventions and the stability condition, the formula in Theorem 4.3.6 of \cite{Lada}. This is  one of the equivalent descriptions of an $IBL_\infty$-algebra.
	In Baranikov's formalism this equation corresponds to the master equation, in Theorem \ref{THMFeynBaran}, for $L$ given above. 
	\begin{remark} \label{remark_hdo}
		In the above theorem we allow all stable values of $(m, n, \chi)$. In this case the corresponding  $IBL_\infty$-algebras are referred to as ``generalized" ones, cf. \cite{Cieliebak ibl}.
		If we assume only non-zero values of $m$ and $n$ and $m+n>2$, for $G=0$, there is another interpretation \cite{Cieliebak ibl}, \cite{Drummond} of  an $IBL_\infty$-algebra in terms of a ``homological differential operator'', cf. end of the previous section. Obviously, for the Frobenius properad, the respective discussion simplifies a lot.   The assignment $\phi^i \mapsto \partial_{a_i}$, $$\partial_{a_i}a_j - (-1)^{|a_i||a_j|} a_j\partial_{a_i} = \delta_i^j$$ turns the generating element $L\in \tilde P$ into a differential operator on $S(V)$\footnote{$P_+$ as introduced before is only a subspace of $S(V)$, but there is no problem in extending $L$ to the whole symmetric algebra.}, $$L= \sum_{m,n,\chi} \sum_{I.J} \frac{1}{m!n!}f^{\chi,J}_I a_J\frac{\partial}{\partial a^I}.$$ 
		
		Finally note, that the differential $d$ on $S(V)\tp S(V^\#)$ can be thought of as an element in $V\tp V^\#$ and hence as a first order differential operator on $S(V)$ with coefficients linear in $a_i$'s. Obviously, the derivatives $\partial_{a_i}$ have the meaning of the left derivatives $\partial^L_{a_i}$, well known from the supersymmetry literature. 
		
		All in all, on $S(V)$, we have a degree one differential operator  
		$d +L$, squaring to $0$, 
		$$(d +L)\circ (d+L)=(d +L)^2=0.$$ 
		
		The last remark: For a formal definition of an $IBL_\infty$-algebra, one can simply consider any degree one differential operator on $S(V)$ squaring to zero. This would accommodate $IBL_\infty$-algebras within the framework of BV formalism \cite{qocha}.
		
	\end{remark}
	\subsection{\texorpdfstring{$IBA_\infty$-algebras and open-closed $IB$-homotopy algebras}{Loop homotopy algebras1}} \label{SECTLoopHomAlg1}
	
	Here we consider the cases of the open and open-closed Frobenius properads.
	In view of the proof of the above Theorem \ref{THMMarklLoopViaOperads}, the following two theorems are straightforward. Their proofs are rather technical, but can be easily reconstructed by following the proofs of the corresponding theorems for modular operads \cite{DJM}.

	Let us consider $n$ inputs and $m$ outputs distributed over $b=p+q$ boundaries of a genus $g$ 2-dimensional oriented surface. More formally, we have a set of cycles $\{ \cc_1, \cc_2, \ldots ,\cc_{p}, \cd_1, \cd_2, \ldots, \cd_{q}\}$, of respective lengths $(k_1, k_2,\ldots, k_{p}, l_1, l_2, \ldots l_{q})$. That is, $\cc_1=\cyc{j_{1}\cdots j_{k_1}}$,  $\cc_2=\cyc{j_{k_1+1}\cdots j_{k_1+k_2}}, \ldots, \cc_{p}=\cyc{j_{k_1+\ldots k_{{p}-1}+1}\cdots j_{m}}$ being  cycles in $[m]$ and similarly $\cd_1=\cyc{i_{1}\cdots i_{l_1}}$,  $\cd_2=\cyc{i_{l_1+1}\cdots i_{l_1+l_2}}, \ldots \cd_{q}=\cyc{i_{l_1+\ldots l_{q-1}+1}\cdots i_{n}}$, being  cycles in $[n]$. Now, let each of the indices $i_1,\ldots i_n$ and $j_1,\ldots j_m$ take values in the set $[\mbox{dim}V]$ and group them into respective multi-indices 
	$$I := i_{1}\cdots i_{l_1} | i_{l_1+1}\cdots i_{l_1+l_2}|\cdots|i_{l_1+\ldots l_{q-1}+1}\cdots + i_{n}$$ 
	$$J:=j_{1}\cdots j_{k_1} | j_{k_1+1}\cdots j_{k_1+k_2}|\cdots|j_{k_1+\ldots k_{p-1}+1}\cdots +j_{m}.$$
	
	We will use the following, hopefully self-explanatory, notation for these indices: 
	$I=I_1|I_2|\cdots|I_{q}$ and similarly for $J$. Concerning the coinvariants
	(\ref{EQFirstIdentification}), consider elements in the tensor algebra $T(V)\otimes T(V^\#)$ of the form $a_{J_1|J_2|\cdots |J_{p}}\tp \phi^{I_1|I_2|\cdots |I_{q}}$, where we identify, up to the corresponding  Kozsul sign,  tensors which differ by cyclic permutations of outputs/inputs within the boundaries, i.e. within the individual multi-indices $I_i$ and $J_j$ and also under permutations of output/input boundaries, i.e. independent permutations of multi/indices
	$(I_i)$ and $(J_i)$. We will denote the subspace of $T(V)\otimes T(V^\#)$ spanned by these elements as $T^{\mbox{cyc}}(V)\otimes T^{\mbox{cyc}}(V^\#)$.
	Further, consider coefficients $f^{(g,p,q){J_1|J_2|\cdots|J_{p}}}_{I_1|I_2|\cdots| I_{q}}$ possessing the corresponding invariance, up the Koszul sing, under  cyclic permutations of outputs/inputs within the boundaries and also under independent permutations of output/input boundaries. 
	
	Put,
	\begin{equation}\label{IBA}
	L = \sum_{p,q,g} \sum_{\substack{I_{1}|I_2|\cdots|I_{q}\\J_1|J_2|\cdots|J_{p}}} \frac{1}{p!q!\prod'_s l_{s}k_{s}} f^{(g,p,q){J_1|J_2|\cdots|J_{p}}}_{I_1|I_2|\cdots| I_{q}} a_{J_1|J_2|\cdots |J_{p}}\tp \phi^{I_1|I_2|\cdots |I_{q}},
	\end{equation}
	where $\prod'$ is the product of nonzero $l_s$'s and $k_s$'s and where $I_s$ runs over all elements of $[\mbox{dim}V]^{\times l_s}$ and similarly $J_s$ runs over all elements of $[\mbox{dim}V]^{\times k_s}$.
	Also, we included the differential into $L$ as an element corresponding to the cylinder with one input and one output.

	\begin{theorem} \label{THMMarklLoopViaOperads1}
		Algebra over  the cobar complex $\mathcal{COF}$ is described by a degree one element $L$ (\ref{IBA}) of $T^{\mbox{cyc}}(V)\otimes T^{\mbox{cyc}}(V^\#)$ such that $L\circ L =0.$
	\end{theorem}
	\begin{remark}\label{remark_hdo1}
		A remark completely analogous to the above Remark \ref{remark_hdo} can be made. In particular, we can think of $\phi_i$ as being represented by a ``left'' derivative $\partial^L_{a_i}$. This is possible because in any monomial of the form $a_{J_1|J_2|\cdots |J_{p}}$ one can always get any of the variables $a_{j_k}$ to the left by a permutation of boundaries and a cyclic permutation within the respective boundary. Hence, if we consider for {\it{any}} collection of multi-indicies ${J_1|J_2|\cdots |J_{p}}$ the tensor product $V^{\otimes J_1}\otimes \ldots \otimes V^{\otimes J_p}$ modulo the respective symmetry relations, on the direct product over all such multi-indices, we have again a homological differential operator $L$.
	\end{remark}

	Finally, let us concern the cobar complex $\mathcal{COCF}$ of the two-colored properad  $\mathcal{OCF}$.
	To describe coinvarinats, consider elements of  $T^{\mbox{cyc}}(V_{\underline{o}})\otimes S(V_{\underline{c}}) \otimes T^{\mbox{cyc}}(V^\#_{\underline{o}})\otimes S(V^\#_{\underline{c}})$ of the form $a_{J_1|J_2|\cdots |J_{p};J}\tp \phi^{I_1|I_2|\cdots |I_{q};I}$
	where $a_{J_1|J_2|\cdots |J_{p};J}:= a_{J_1|J_2|\cdots |J_{p}}\tp a_J$ and $\phi^{I_1|I_2|\cdots |I_{ q};I}:=\phi^{I_1|I_2|\cdots |I_{q}}\tp \phi^I$.  Correspondingly, consider the coefficients $f^{(g,p,q){J_1|J_2|\cdots|J_{p}};J}_{I_1|I_2|\cdots| I_{q};I}$ with the obvious symmetry properties. 
	Put,
	\begin{equation}
	\label{OC}
	L = \sum_{m,n, p,q,g} \sum_{\substack{I_{1}|I_2|\cdots|I_{q};I\\J_1|J_2|\cdots|J_{p};J}} \frac{1}{m!n! p!q!\prod'_s l_{s}k_{s}} f^{(g,p,q){J_1|J_2|\cdots|J_{p};J}}_{I_1|I_2|\cdots| I_{q};I} a_{J_1|J_2|\cdots |J_{p};J}\tp \phi^{I_1|I_2|\cdots |I_{q};I},
	\end{equation}
	where, as before, $\prod'$ is the product of nonzero $l_s$'s and $k_s$'s and where $I_s$ runs over all elements of $[\mbox{dim}V_{\underline o}]^{\times l_s}$ and  $J_s$ runs over all elements of $[\mbox{dim}V_{\underline o}]^{\times k_s}$. The closed multi-index $I$ runs over all elements of $[\mbox{dim}V_{\underline c}]^{\times m}$, similarly $J$ runs over all elements of $[\mbox{dim}V_{\underline c}]^{\times n}$.
	Also, we included the open and closed differentials into $L$ as elements corresponding to the cylinder with one input and one output and to sphere with one input and one output, respectively.
	
	\begin{theorem} 
		Algebra over  the cobar complex $\mathcal{COCF}$ is described by degree one element $L$ (\ref{OC}) of  $T^{\mbox{cyc}}(V_{\underline{o}})\otimes S(V_{\underline{c}}) \otimes T^{\mbox{cyc}}(V^\#_{\underline{o}})\otimes S(V^\#_{\underline{c}})$, such that $L\circ L =0.$
	\end{theorem}
	
	Finally, remarks \ref{remark_hdo} and \ref{remark_hdo1} apply correspondingly.

	\newpage
	\vspace{2cm}\noindent {\bf Acknowledgements}:
	The research of M.D. and B.J. was supported by grant GA\v CR 18-07776S. The research of L.P. was supported by the grant SVV-2017-260456 and GAUK 544218.  B.J. wants to thank MPIM in Bonn for hospitality. We also thank to the anonymous referee for very useful comments.
	

\end{document}